\documentclass[12pt]{amsart}
\usepackage{amsmath,amssymb,amsfonts,amsthm,amscd,indentfirst}
\usepackage[a4paper, left=2.5cm, right=2.5cm, top=2cm, bottom=2cm]{geometry}
\usepackage{amsmath,amssymb,amsfonts,amsthm,amscd}
\usepackage{indentfirst}
\usepackage{hyperref}
\usepackage[numbers,sort]{natbib}

\hypersetup{
colorlinks=true,
linkcolor=black
}

\numberwithin{equation}{section}

\newtheorem{proposition}{Proposition}[section]
\newtheorem{theorem}[proposition]{Theorem}
\newtheorem{lemma}[proposition]{Lemma}

\newtheorem{remark}[proposition]{Remark}
\newtheorem{example}[proposition]{Example}
\newtheorem{definition}[proposition]{Definition}

\newtheorem{notation}[proposition]{Notation}

\newcommand{\ttt}{\theta}
\newcommand{\lam}{\lambda}
\newcommand{\lamm}{\Lambda}

\newcommand{\wq}{\infty}
\newcommand{\te}{\text}

\newcommand{\rr}{\mathbb{R}}

\newcommand{\pa}{\partial}

\newcommand{\p}{\phi}

\newcommand{\es}{\varepsilon}
\newcommand{\e}{\eta}
\newcommand{\dd}{\Delta}
\newcommand{\de}{\delta}

\newcommand{\al}{\alpha}
\newcommand{\bb}{\beta}
\newcommand{\ga}{\gamma}
\newcommand{\gaa}{\Gamma}

\newcommand{\ooo}{\Omega}
\newcommand{\si}{\sigma}

\newcommand{\z}{\left}
\newcommand{\x}{\right}

\newcommand{\ka}{\kappa}

\allowdisplaybreaks[4]
 
\usepackage{graphicx}
\begin{document}
\title[Pogorelov estimates and rigidity theorems]{The Pogorelov estimates for the sum Hessian equation with rigidity theorem and parabolic versions}
\author[Weizhao Liang]{Weizhao Liang}
\address{Department of Mathematics, University of Science and Technology of China, Hefei, China.}
\email{Lwz740@mail.ustc.edu.cn}
\author[Jin Yan]{Jin Yan}
\address{Department of Mathematics, University of Science and Technology of China, Hefei, China.}
\email{yjoracle@mail.ustc.edu.cn}
\author[Hua Zhu]{Hua Zhu}
\address{School of Mathematical and Physics, Southwest University of Science and Technology, Mianyang, 621010, Sichuan Province, China.}
\email{zhuhmaths@mail.ustc.edu.cn}

\begin{abstract}
    In this paper, we primarily study the Pogorelov-type $C^2$ estimates for $(k-1)$-convex solutions of the sum Hessian equation under the assumption of semi-convexity, and apply these estimates to obtain a rigidity theorem for global solutions satisfying the corresponding conditions. Furthermore, we investigate the Pogorelov estimates and rigidity theorems for solutions to the parabolic versions under similar conditions. These results extend the works of \cite{He-Sheng-Xiang-Zhang-2022-CCM} and \cite{Liu-Ren-2023-JFA}.
\end{abstract}


\maketitle

\section{Introduction}\label{sec 1}
\setcounter{equation}{0}

In this paper, we consider the following Dirichlet problem for the elliptic sum Hessian equation:
\begin{equation}\label{main eq}
    \left\{\begin{aligned}
    &\sigma_{k}(\lam(D^2u))+\al \sigma_{k-1}(\lam(D^2u))  =f(x, u, D u), & &  \text {in } \Omega, \\
    &u  =0, & & \text {on } \partial \Omega,
    \end{aligned}\right.
\end{equation}
where $\ooo\subset \rr^n$ is a bounded domain, $u$ is an unknown function defined in $\overline{\Omega}$, and $\al$ is a given positive constant. We denote by $D u$ and $D^2 u$ the gradient and Hessian of $u$, respectively.  Additionally, we assume that $f \geq  C_{0}>0$. Let $\si_k(D^2 u) = \si_k(\lam(D^2 u))$ denote the $k$-th elementary symmetric function of the eigenvalues of the Hessian matrix $D^2 u$. Specifically, for $\lam = (\lam_1,\cdots,\lam_n) \in \rr^n$, we have 
$$
    \si_k(\lam)=\sum_{1\leq i_1<\cdots<i_k\leq n} \lam_{i_1} \cdots \lam_{i_k}.
$$

As is well-known, the following classic $k$-Hessian equation
\begin{equation}\label{1.2}
    \sigma_k(D^2u)=f(x,u,Du),\quad x \in\Omega
\end{equation}
plays a significant role in the study of fully nonlinear partial differential equations and geometric analysis. This equation is closely related to numerous geometric problems (e.g., \cite{BK,CNS3,CNS4,CY,GLL,Mei-Zhu-2024-JGA,GLM,O,P3,TW}). Equation \eqref{main eq} generalizes \eqref{1.2} by combining the $k$-Hessian operators linearly. These types of equations have been extensively studied and applied in geometry. For instance, Harvey and Lawson \cite{HL80} studied special Lagrangian equations within this framework. Krylov \cite{Kry} and Dong \cite{Dong} investigated similar equations and derived curvature estimates based on the concavity of the operators. Guan and Zhang \cite{GZ} analysed curvature estimates for related equations where the right-hand side does not depend on the gradient but involves coefficients tied to hypersurface positions. Additionally, hyperbolic geometry problems reduce to equations of this type \cite{EGM}.

A central problem in studying Hessian equations like \eqref{1.2} is establishing $C^2$ estimates. For relevant work, see \cite{CNS1,CNS3,CW,B,GLL,GLM,GRW}. Notably, when the right-hand side depends on the gradient, obtaining $C^2$ estimates remains challenging \cite{GLL}.

In this paper, we focus on Pogorelov-type $C^2$ estimates for the Dirichlet problem \eqref{main eq}. These estimates, which involve interior $C^2$ bounds influenced by boundary conditions, were first developed by Pogorelov \cite{P3} for Monge-Ampère equations and  play a crucial role in studying the regularity of fully nonlinear equations. For instance, Pogorelov-type $C^2$ estimates have been used to analyse degenerate Monge-Ampère equations \cite{Blocki,Savin}. For $k$-Hessian equations \eqref{1.2} with $f=f(x,u)$, Chou and Wang \cite{CW} established Pogorelov-type $C^2$ estimates for $k$-convex solutions, yielding
\begin{equation*}
    (-u)^{1+\delta}\Delta u\leq C.
\end{equation*} 
Sheng, Ubas, and Wang \cite{SUW} extended such estimates to curvature equations for graphic hypersurfaces, including Hessian equations. For $f=f(x,u,Du)$, Li, Ren, and Wang \cite{LRW2} developed new techniques to drop the small $\es$ and proved Pogorelov-type $C^2$ estimates for $(k+1)$-convex solutions of \eqref{1.2}:
\begin{equation*}
    (-u)\Delta u\leq C.
\end{equation*}
Recently, Tu \cite{Tu-2024-arxiv} derive the Pogorelov-type for \eqref{1.2} under $k$-convex and semi-convexity assumptions, following ideas from Lu \cite{Lu-2023-CVPDE}.

For sum Hessian equations, Li, Ren, and Wang \cite{LRW1} showed that the operator $\sigma_k+\al\sigma_{k-1}$ is elliptic in the convex-cone
$$
\tilde\Gamma_k=\Gamma_{k-1}\cap\{\lambda|\sigma_k(\lambda)+\alpha\sigma_{k-1}(\lambda)>0\},
$$
which is the admissible set for the equations in \eqref{main eq}. Here, $(k-1)$-convex solutions and $f>0$ ensure ellipticity. Then Liu and Ren \cite{Liu-Ren-2023-JFA} established Pogorelov-type $C^{2}$ estimates for $k$-convex solutions of \eqref{main eq}. 

From an analytical perspective, a natural question is whether the $k$-convexity assumption in their paper can be relaxed. In this paper, we address this question by establishing Pogorelov-type $C^2$ estimates under $(k-1)$-convexity and semi-convexity assumptions. Additionally, we derive a parabolic version of the problem and a rigidity theorem. Building on \cite[Lemma 2.7 (a)]{Liu-Ren-2023-JFA}, our assumptions are weaker than $k$-convexity.

First, we introduce definitions of certain convexity and growth conditions for the solution.

\begin{definition}\label{def sigma_k}
    Let $\ooo \subset \mathbb{R}^n$ be a domain and $u\in C^2(\ooo)$. 
    \par
    $(1)$ We say that $u$ is $k$-convex if the eigenvalues $\lambda(x) = (\lambda_1(x), \cdots, \lambda_n(x))$ of the Hessian $D^2 u(x)$ belong to $\Gamma_k$ for all $x \in \ooo$, where $\Gamma_k$ is the Garding's cone
    $$
    \Gamma_k = \{ \lambda \in \mathbb{R}^n \mid \sigma_i(\lambda) > 0, \, i = 1, \dots, k \}.
    $$
    \par
    $(2)$ We say that $u$ is semi-convex if there exists a constant $K_0>0$, such that
    $$
    \lam_i(x)> -K_0, \quad 1\leq i\leq n, \quad \forall \ x \in \ooo.
    $$
    \par
    $(3)$ We say that $u$ satisfies the quadratic growth condition if there condition positive constants $a,b$ and a sufficiently large $R$ such that
    \begin{equation}\label{growth condition}
        u(x)\geq a|x|^2-b,\quad |x|\geq R.
    \end{equation}
    \par
    $(4)$ We denote $\tilde{\gaa}_{k}$ as: 
    \begin{equation*}
        \tilde\Gamma_k=\Gamma_{k-1}\cap\{\lambda|\sigma_k(\lambda)+\alpha\sigma_{k-1}(\lambda)>0\}.
    \end{equation*}
\end{definition}

For the sum Hessian equation, we can obtain the following Pogorelov-type estimates.

\begin{theorem}\label{thm main thm: elliptic}
    Let $\ooo\subset \rr^n$ be a bounded domain and $f(x, u, p) \in C^{1,1}(\overline{\ooo} \times \mathbb{R} \times \mathbb{R}^n)$ be a positive function. Assume $u \in C^4(\ooo) \cap C^{0,1}(\overline{\ooo})$ is a $(k-1)$-convex and semi-convex solution to the Dirichlet problem \eqref{main eq} with $\alpha > 0$. Then there exist two positive constant $\ga_0$ and $C$ such that
    \begin{equation}\label{eq main result}
        (-u)^{\ga_0} \, \dd u \leq C,
    \end{equation}
    where $\ga_0$, $C$ depend on $n$, $k$, $\inf f$, $\|f\|_{C^{1,1}}$, $\|u\|_{C^1}$, and $\operatorname{diam}(\ooo)$.
\end{theorem}

An important application of the interior $C^2$ estimates is to get rigidity theorems for certain equations. Consider the following $k$-Hessian equation in $n$-dimensional Euclidean spaces:
\begin{equation}\label{1.5}
    \sigma_k(D^2u)=1,
\end{equation}
Chang and Yuan \cite{CY} posed the following question: Are the entire solutions of \eqref{1.5} with lower bound only quadratic polynomials?

For $k=1$, equation \eqref{1.5} is linearand the result directly follows from the Liouville property of harmonic functions. For $k=n$, equation \eqref{1.5} corresponds to the Monge-Ampère equation, a well-known result in geometry. The works of Jörgens \cite{K}, Calabi \cite{E}, and Pogorelov \cite{P1, P3} established that any entire strictly convex solution to this equation must be a quadratic polynomial. Cheng and Yau \cite{CYa} provided an alternative geometric proof, and in 2003, Caffarelli and Li \cite{CL} extended the classical result of Jörgens, Calabi, and Pogorelov.

For general $1<k<n$, there are few related results. For $k=2$, Chang and Yuan \cite{CY} proved that if the Hessian satisfies 
$$
D^2u\geq \z[\delta-\sqrt{\frac{2n}{n-1}}\x] I,
$$
for any $\delta>0$, then the any entire solution of the equation \eqref{1.5} must be a quadratic polynomial. For general $k$, Bao, Chen, Guan, and Ji \cite{BCGM} showed that strictly convex entire solutions of \eqref{1.5} with quadratic growth are quadratic polynomials. Later, Li, Ren, and Wang \cite{LRW2} extended this result by relaxing the strictly convex condition to $(k+1)$-convex solutions, and proved that entire $(k+1)$-convex solutions of \eqref{1.5} with quadratic growth are also quadratic polynomials. However, in 2016, Warren \cite{Warren} constructed examples showing that \eqref{1.5} admits non-polynomial entire $k$-convex solutions when $n\geq2k-1$.

As a consequence of Theorem \ref{thm main thm: elliptic}, we can derive a rigidity theorem for the $(k-1)$-convex solutions to the sum Hessian equation with semi-convex condition:
\begin{equation}\label{main eq 1}
    \si_k(D^2 u)+\al \si_{k-1}(D^2 u)=1, \quad \te{in} \ \rr^n,
\end{equation}
where $\al>0$. The rigidity theorem states:
\begin{theorem}\label{thm rigidity theorem 1}
    Let $u \in C^4(\ooo) \cap C^{0,1}(\overline\ooo)$ be a $(k-1)$-convex and semi-convex solution to \eqref{main eq 1} with quadratic growth. Then $u$ must be a quadratic polynomial.
\end{theorem}

Next, we study the following parabolic sum Hessian equation: 
\begin{equation}\label{main parabolic eq}
    \left\{\begin{aligned}
    &-u_t \cdot \z(\sigma_{k}(D^2u)+\al \sigma_{k-1}(D^2u)\x)  =f, & &\text {in } \mathcal{O}, \\
    &u  =0, & & \text {on } \partial \mathcal{O},
    \end{aligned}\right.
\end{equation}
where $\al>0$ and $\mathcal{O} \subset \rr^n \times (-\wq,0]$ is a bounded domain. For $t\leq 0$, we denote 
$$
\mathcal{O}(t):=\{x \in \rr^n|(x,t)\in \mathcal{O}\}, \quad \te{and} \quad  \underline{t}:=\inf\{t\leq 0|\mathcal{O}(t)\ne \emptyset\}.
$$ 
The parabolic boundary $\pa \mathcal{O}$ is defined by
$$
\pa \mathcal{O}=(\overline{\mathcal{O}(\underline{t})}\times \{\underline{t}\})\cup \bigcup_{t\leq 0}(\pa \mathcal{O}\times \{t\}).
$$

As far as we know, rigidity theorems for parabolic fully nonlinear equations are most well-known in the context of parabolic Monge-Ampère equations. Guti$\acute{e}$rrez and Huang \cite{GH} extended theorem of J$\ddot{o}$rgens, Calabi, and Pogorelov to parabolic Monge-Amp\`ere equations. Xiong and Bao \cite{XB11} proved rigidity theorems for 
$$
u_t=(\det D^2u)^{1/n}.
$$ 
Zhang, Bao and Wang \cite{ZBW} extended Caffarelli-Li theorem \cite{CL} to the parabolic Monge-Amp\`ere equation 
$$
-u_t \cdot \det D^2u =f,
$$ 
and further investigated the asymptotic behavior of solutions at infinity. Their approach also provides a framework for treating other parabolic Monge-Ampère equations. For general $k$, Nakamori and Takimoto \cite{NAKAMORI2015211} proved  rigidity theorems for convex-monotone solutions to parabolic $k$-Hessian equations. Subsequently, He, Sheng, and Xiang \cite{He-Sheng-Xiang-Zhang-2022-CCM} extended this result to $(k+1)$-convex-monotone solutions. Recently, Bao, Qiang, Tang, and Wang \cite{Bao-2023-CPAA} refined this result, extending it to $k$-convex-monotone solutions.

Inspired by the results in \cite{He-Sheng-Xiang-Zhang-2022-CCM} and \cite{Liu-Ren-2023-JFA}, we extend their work to derive Pogorelov estimates for the parabolic sum Hessian equation. To this end, we first introduce the notion of $(k-1)$-convex-monotone solutions.
\begin{definition}
    A function $u \in C^2\z(\rr^n \times (-\wq,0]\x)$ is called $(k-1)$-convex-monotone if $\lam(D^2 u)\in \gaa_{k-1}$ and $u$ is non-increasing in $t$.
\end{definition}
  
The following theorem is the parabolic version of \cite[Theorem 1.2 (b)]{Liu-Ren-2023-JFA}, and is similar to \cite{Bao-2023-CPAA}.
\begin{theorem}\label{thm main thm: parabolic 1}
    Let $f(x, u,t) \in C^{1,1}(\overline{\mathcal{O}} \times \mathbb{R}\times [\underline{t},0] )$ be a positive function. Assume $u \in C^4(\mathcal{O}) \cap C^{0,1}(\overline{\mathcal{O}})$ is a $(k-1)$-convex-monotone solution to \eqref{main parabolic eq} satisfying $0<m_1\leq -u_t\leq m_2$ and $\alpha > 0$. Then for any $\de>0$ small, there exists a positive constant $C$ such that
    \begin{equation}\label{eq main result: parabolic}
        (-u)^{1+\de} \, \dd u \leq C,
    \end{equation}
    where $C$ depends on $n$, $k$, $m_1$, $m_2$, $\inf f$, $\|f\|_{C^2}$, $\sup |u|$, and $\operatorname{diam}(\mathcal{O}(0))$.
\end{theorem}

When $f$ depends on the gradient, we have the following conclusion:
\begin{theorem}\label{thm main thm: parabolic 2}
    Let $f(x, u,p,t) \in C^{1,1}(\overline{\mathcal{O}} \times \mathbb{R}\times \rr^n \times [\underline{t},0] )$ be a positive function. Assume $u \in C^4(\mathcal{O}) \cap C^{0,1}(\overline{\mathcal{O}})$ is a $(k-1)$-convex-monotone solution to \eqref{main parabolic eq} satisfying $0<m_1\leq -u_t\leq m_2$ and $\alpha > 0$. We also assume $u$ satisfies smei-convexity. Then there exist two positive constant $\ga_0$ and $C$ such that
    \begin{equation}\label{eq main result: parabolic 2}
        (-u)^{\ga_0} \, \dd u \leq C,
    \end{equation}
    where $\ga_0$, $C$ depend on $n$, $k$, $m_1$, $m_2$, $\inf f$, $\|f\|_{C^2}$, $\|u\|_{C^1}$, and $\sup |u|$, and $\operatorname{diam}(\mathcal{O}(0))$.   
\end{theorem}

As in the elliptic case, by applying a similar method as in Theorem \ref{thm main thm: parabolic 2}, we derive the rigidity theorem for  $(k-1)$-convex solution of the parabolic sum Hessian equation:
\begin{equation}\label{main eq 2}
    -u_t \cdot\z(\sigma_{k}(D^2u)+\al \sigma_{k-1}(D^2u)\x)  =1,  \quad \text {in } \rr^{n}\times(-\infty,0],
\end{equation}   
\begin{theorem}\label{thm rigidity theorem 2}
    Let $u \in C^4(\mathcal{O}) \cap C^{0,1}(\overline{\mathcal{O}})$ be a $(k-1)$-convex and semi-convex solution to \eqref{main eq 2} with $\alpha > 0$. We assume that $u(x,0)$ satisfies a quadratic growth condition and $0<m_1\leq -u_t\leq m_2$. Then $u$ has the form 
    $$
    u(x,t)=-mt+p(x),
    $$ 
    where $m>0$ is a constant and $p(x)$ is a quadratic polynomial.
\end{theorem}
Finally, using the idea of  Warren \cite{Warren}, we present an example of a non-polynomial entire $(k-1)$-convex solution, which does not satisfy the assumptions of Theorem\ref{thm rigidity theorem 1}:

\begin{example} When $n=3$, the $1$-convex function
$$
u(x,y,t)=\frac{e^{4t}-1}{4}(x^2+y^2)+\frac{1}{16}(\frac{7e^{-4t}}{4}-\frac{e^{4t}}{4}-4t^2)
$$
solves the equation
$$
\sigma_2(D^2u)+\sigma_1(D^2u)=1.
$$
\end{example}

This paper is organized as follows: In section \ref{sec 2}, we present some basic lemmas. In section \ref{sec 3}, we prove Theorem \ref{thm main thm: elliptic}. In section \ref{sec 4}, we prove Theorem \ref{thm main thm: parabolic 1} and \ref{thm main thm: parabolic 2}. Finally, in Section \ref{sec 5}, we prove two rigidity results, namely, Theorems \ref{thm rigidity theorem 1} and \ref{thm rigidity theorem 2}. Throughout this paper, we always assume $\lam_1$ to be large enough.

\section{Preliminaries}\label{sec 2}
\setcounter{equation}{0}

In this section, we give the preliminary knowledge related to our theorems and their proofs. For convenient, we denote some notations as same as \cite{Liu-Ren-2023-JFA}:
\begin{notation}
    Let $\lam=(\lam_1,\cdots,\lam_n)\in \rr^n$. For any $1\leq p,q \leq n$, we define
    \begin{enumerate}
        \item[$(1)$] $S_k(\lam):=\si_k(\lam)+\al \si_{k-1}(\lam)$, for $\al >0$,
        \item[$(2)$] $S_k^{pp}(\lam):=\frac{\pa S_k(\lam)}{\pa \lam_p}=\si_{k-1}(\lam|p)+\al\si_{k-2}(\lam|p)=S_{k-1}(\lam|p)$,
        \item[$(3)$] $S_k^{pp,qq}(\lam):=\frac{\pa S_k(\lam)}{\pa \lam_p \pa \lam_q}=S_{k-2}(\lam|pq)$, and $S_k^{pp,pp}(\lam)=0$,
        \item[$(4)$] $S_k(\lam)=\lam_p S_{k-1}(\lam|p)+S_k(\lam|p)$,
        \item[$(5)$] $\sum\limits_{p=1}^n S_k(\lam|p)=(n-k)S_{k}(\lam)+\al\si_{k-1}(\lam)$, 
        \item[$(6)$] $\sum\limits_{p=1}^{n}\lam_p S_{k-1}(\lam|p)=kS_k(\lam)-\al\si_{k-1}(\lam)$.
    \end{enumerate}
\end{notation}

\begin{lemma}[See \cite{Ball}]\label{lem 2.2}
    Denote by $Sym(n)$ the set of all $n\times n$ symmetric matrices.
    Let $F$ be a $C^2$ symmetric function defined in some open subset
    $\Psi \subset Sym(n)$. At any diagonal matrix $A\in\Psi$ with distinct eigenvalues, let $\ddot{F}(B,B)$ be the second derivative of $C^2$ symmetric function $F$ in direction $B \in Sym(n)$, then
    \begin{align*}
        \ddot{F}(B,B) =  \sum_{j,k=1}^n {\ddot{f}}^{jk}
        B_{jj}B_{kk} + 2 \sum_{j < k} \frac{\dot{f}^j -
        \dot{f}^k}{{\kappa}_j - {\kappa}_k} B_{jk}^2.
    \end{align*}
\end{lemma}

\begin{lemma}[See \cite{LRW1,Liu-Ren-2023-JFA}]\label{lem 2.3}
    Assume that $\lambda=(\lambda_1,\cdots,\lambda_n)\in\tilde\Gamma_k$, $1\leq k\leq n$, $\lambda_1\geq\cdots\geq\lambda_n$, then
    \par
    $(a)$ The set $\tilde{\gaa}_{k}$ is a convex cone. In $\tilde{\gaa}_k$, equation \eqref{main eq} with $1\leq k\leq n $ and $\al >0$  is elliptic.
    \par
    $(b)$ For any $1\leq j\leq k-1$, there exists a positive constant $\theta$ depending on $n,k$, such that
    \begin{equation}\label{eq 2.2}
        S_k^{jj}(\lam)\geq \frac{\theta S_k(\lam)}{\lam_j}.
    \end{equation}
    \par 
    $(c)$ $S_k^2(\lam)-S_{k-1}(\lam)S_{k+1}(\lam)\geq 0$.
    \par
    $(d)$ $S_k^{\frac{1}{k}}(\lam)$ and $\z[\frac{S_k}{S_l}\x]^{\frac{1}{k-l}}(\lam)$ is concave in $\tilde\gaa_k$ for $1\leq l < k\leq n$.
    \par
    $(e)$ Let $k>l$, $\tau=\frac{1}{k-l}$, then 
    \begin{align*}
        -&\frac{S_k^{pp,qq}}{S_k}u_{ppj}u_{qqj}+\frac{S_l^{pp,qq}}{S_l}u_{ppj}u_{qqj}\\
        &\geq \z(\frac{(S_k)_j}{S_k}-\frac{(S_l)_j}{S_l}\x)\z(
        (\tau-1)\frac{(S_k)_j}{S_k}-(\tau+1)\frac{(S_l)_j}{S_l}\x).\nonumber
    \end{align*}
    Furthermore, for any $\delta>0$, we have
    \begin{align}\label{2.2}
        -&S_k^{pp,qq}u_{ppj}u_{qqj} +\z(1-\tau+\frac{\tau}{\de}\x)\frac{(S_k)_j^2}{S_k}\\
        &\geq S_k(1+\tau-\de\tau) \z[\frac{(S_l)_j}{S_l}\x]^2-\frac{S_k}{S_l} S_l^{pp,qq} u_{ppj}u_{qqj}.\nonumber
    \end{align}
\end{lemma}
\begin{remark}
    In fact, conclusion $(c)$ was proved by Ren in \cite{Ren-2024-IMRN}, and does not require the condition $\lam\in \tilde{\gaa}_k$.
\end{remark}

\begin{lemma}\label{lem 2.5}
    Let $\lam$ be as defined in Lemma $\ref{lem 2.3}$. Assume that $N_0\leq S_k\leq N_1$ and $\lam_n\geq -K_0$ for positive constants $N_0$, $N_1$, $K_0$. Then, for sufficiently large $\lambda_{1}$, there exists a positive constant $c_0>0$ such that
    \par
    $(a)$ $\si_{k-1}(\lam)\geq c_0(n,k,N_0)$.
    \par
    $(b)$ $|\lam_k| \leq c_0(n,k,N_1)K_0$.
    \par
    $(c)$ $\sum\limits_{i} S_k^{ii}\geq c_0(n,k,N_0) \si_1^{\frac{1}{k-2}}$.
\end{lemma}
\begin{proof}
    (a) If $\si_k>0$, by the Newton-MacLaurin inequality, we have
    $$      
    \si_{k-1}\geq c_1 \si_{k}^{\frac{k-1}{k}}, 
    $$
    which implies 
    $$
    N_0\leq S_k \leq \z(\frac{\si_{k-1}}{c_1}\x)^{\frac{k}{k-1}}+\al \si_{k-1}\leq C(n,k) \max\z\{\si_{k-1}^{\frac{k}{k-1}},\si_{k-1}\x\}.
    $$

    If $\si_k\leq 0$, then $N_0\leq S_k \leq \al \si_{k-1}$. Hence, in both cases, $\si_{k-1}\geq c_0$ for some positive constant $c_0$.
    \par
    (b) Assume the contrary, that $\lam_k \gg K_0$. Then
    $$
    N_1 \geq S_k=\si_k+\al \si_{k-1}\geq \lam_1 \cdots \lam_{k-1}\z[\lam_k-C(n,k)K_0\x] \gg 1.
    $$
    This leads to a contradiction. Therefore, $\lam_k\leq c_0 K_0$ for some constant $c_0>0$.
    \par
    (c) From the definition of $S_k^{ii}$, we have 
    $$
    \sum_i S_k^{ii}=(n-k+1)\si_{k-1}+\al (n-k+2)\si_{k-2}.
    $$
    By (a) and the Newton-MacLaurin inequality, 
    $$
    \sum_i S_k^{ii}\geq \al(n-k+2) \si_{k-2}\geq   C \si_1^{\frac{1}{k-2}} \si_{k-1}^{\frac{k-3}{k-2}} \geq  c_0 \si_1^{\frac{1}{k-2}}.
    $$
\end{proof}

\section{Pogorelov estimates for $(k-1)$-convex solutions}\label{sec 3}
\setcounter{equation}{0}

In this section, we establish Pogorelov-type $C^2$ estimates for $(k-1)$-convex and semi-convex solutions to the Dirichlet problem \eqref{main eq}.  

Let $S_{k}=\sigma_{k}+\alpha \sigma_{k-1}$, where $\al>0$. We now consider the corresponding test function:
$$
    \phi=M \log (-u)+\log P_{m}+\frac{A}{2}|D u|^{2}+\frac{B}{2} |x|^{2},
$$ 
where $P_{m}=\sum_{i=1}^{n} \kappa_{i}^{m}$, $\kappa_{i}=\lambda_{i}+K_{0}>0$, $K_0$ is defined in Definition \ref{def sigma_k}, and $M, m, A, B>0$ are parameters to be determined later.

Assume that $\p$ attains its maximum at $x_{0}$. Furthermore, suppose that the Hessian of $u$ at $x_0$ is given by $D^{2} u\left(x_{0}\right)=\operatorname{diag}\left(u_{11}\left(x_{0}\right), \cdots, u_{n n}\left(x_{0}\right)\right)$, with $u_{11}\left(x_{0}\right) \geq \cdots \geq  u_{n n}\left(x_{0}\right)$. 
In the remainder of the proof, all computations will be carried out at $x_0$.
\begin{equation}\label{eq first derivative of test function}
    0=\phi_{i}=M \frac{u_{i}}{u}+\frac{m \sum_{j} \ka_{j}^{m-1}}{P_{m}} u_{j j i}+A u_{i} u_{i i}+B x_{i} ,
\end{equation}
\begin{align}
    0 \geq \phi_{i i}=& M \frac{u_{i i}}{u}-M \frac{u_{i}^{2}}{u^{2}}+m(m-1) \frac{\sum_{j} \kappa_{j}^{m-2}}{P_{m}} u_{j j i}^{2}-\frac{m^{2}\left(\sum_{j} \kappa_{j}^{m-1} u_{j j i}\right)^{2}}{P_{m}^{2}} \label{eq second derivative of test function} \\
    & +\frac{m \sum_{j} \kappa_{j}^{m-1}}{P_{m}}\left[u_{j j i i}-2 \sum_{p \neq j} \frac{u_{p j i}^{2}}{u_{p p}-u_{j j}}\right]+A u_{i i}^{2}+A u_{s} u_{ s i i}+B. \nonumber
\end{align}

By Lemma \ref{lem 2.3} (a), $\{S_{k}^{i j}\}$ is elliptic. Therefore, contracting \eqref{eq second derivative of test function} with $S_k^{ii}$, we have
\begin{align*}
    0 \geq S_{k}^{i i} \phi_{i i}  =& M S_{k}^{i i} \frac{u_{i i}}{u}-M S_{k}^{i i}\left(\frac{u_{i}}{u}\right)^{2}+m(m-1)  \frac{\sum_{j}\kappa_{j}^{m-2}}{P_{m}} S_{k}^{i i} u_{j j i}^{2} \\
    & -\frac{m^{2}}{P_{m}^{2}} S_{k}^{i i}\left(\sum_{j} \kappa_{j}^{m-1} u_{j j i}\right)^{2}+\frac{m}{P_{m}} S_{k}^{i i} \sum_{j}\kappa_{j}^{m-1} u_{j j i i} \\
    & +\frac{m}{P_m}S_{k}^{i i}  \sum_{p \neq j} \frac{\kappa_{p}^{m-1}-\kappa_{j}^{m-1}}{\kappa_{ p}-\kappa_{ j}} u_{p j i}^{2}+A S_{k}^{i i} u_{i i}^{2}+A S_{k}^{i i} u_{s} u_{s i i}+B \sum_{i} S_{k}^{i i}.
\end{align*}

Differentiating equation \eqref{main eq} gives
\begin{align}
    S_{k}^{p q} u_{p q j}& =f_{x_j}+f_{u} u_{j}+  f_{p_j} u_{jj},  \label{First derivation of the main eq} \\
    S_{k}^{pq, r s} u_{p q j} u_{r s j}+S_{k}^{p q} u_{p q jj}   =&f_{x_j x_j}+f_{u u} u_{j}^{2}+f_{u} u_{jj}+f_{p_j p_{j}} u_{jj}^{2}+f_{p_s} u_{s jj}\label{second derivation of the main eq}  \\
    & +2 f_{x_j u} u_{j}+2 f_{x_j p_{j}} u_{jj}+2 f_{u p_{j}} u_{j} u_{jj}. \nonumber
\end{align}
Thus
$$
\sum_{i} S_{k}^{i i} u_{j j i i}\geq-C-Cu_{11}^{2}+\sum_{s} f_{p_s} u_{s j j}-S_{k}^{pq,rs} u_{p q j} u_{r s j},
$$
and
\begin{align*}
    \frac{m}{P_{m}} S_{k}^{ii} \sum_{j} \kappa_{j}^{m-1} u_{j j i i}  \geq& \frac{m}{P_{m}} \sum_{j} \kappa_{j}^{m-1}\left[-C-C u_{11}^{2}+\sum_{s} f_{p_j} u_{s j j}\right] \\
    & -\frac{m}{P_{m}} \sum_{j} \kappa_{j}^{m-1} S_{k}^{pq,rs} u_{pqj}u_{rsj}.
\end{align*}

Therefore
\begin{align}
    0 %&\geq M \frac{S_{k}^{i i} u_{i i}}{u}-M \frac{S_{k}^{i i} u_{i}^{2}}{u^{2}}+A S_{k}^{i i} u_{i i}^{2}+A S_{k}^{i i} u_{s} u_{s i i}+B \sum_{i} S_{k}^{i i}-C\left(u_{11}+\frac{1}{u_{11}}\right) \\
    %&+ \frac{m(m-1)}{P_m} \sum_j \kappa_j^{m-2} S_k^{ii} u_{jji}^2 + \frac{m}{P_m}\sum_j \kappa_j^{m-1} \sum_s f_{p_s}u_{sjj} -\frac{m}{P_m} \sum_j \kappa_j^{m-1} S_{k}^{pq,rs}u_{pqj}u_{rsj} \\
    %&+ \frac{m}{P_m}\sum_{p \neq j} \frac{\kappa_{p}^{m-1}-\kappa_{j}^{m-1}}{\kappa_{p}-\kappa_{j}} u_{pj i}^{2}-\frac{m^{2}}{P_{m}^{2}} \sum_{i} S_{k}^{i i}\left(\sum_{j} \kappa_{j}^{m-1} u_{j j i}\right)^{2} \\ 
    \geq& M \frac{S_{k}^{i i} u_{i i}}{u}-M \frac{S_{k}^{i i} u_{i}^{2}}{u^{2}}+A S_{k}^{i i} u_{i i}^{2}+B \sum_{i} S_{k}^{i i}-Cm\left(1+u_{11}\right) \label{inequality 1} \\
    &+ \underbrace{A S_k^{i i} u_{s} u_{s i i}+\frac{m}{P_{m}} \sum_{s} \sum_{j} \kappa_{j}^{m-1} f_{p_s} u_{s j j}}_{\operatorname{I}} \underbrace{-\frac{m}{P_{m}} \sum_j \kappa_{j}^{m-1} S_{k}^{pq, r s} u_{p q j} u_{r s j}}_{\operatorname{II}} \nonumber\\
    & + \underbrace{\frac{m(m-1)}{P_{m}} \sum_j \kappa_{j}^{m-2} S_{k}^{ii} u_{jji}^2}_{\operatorname{III}} 
    +\underbrace{\frac{m}{P_{m}} \sum_{p\ne j}  S_{k}^{ii}\frac{\kappa_{p}^{m-1}-\kappa_j^{m-1}}{\kappa_p-\kappa_j}u_{p ji}^2 }_{\operatorname{IV}} \nonumber \\
    & \underbrace{-\frac{m^{2}}{P_{m}^{2}} \sum_{i} S_{k}^{i i}\left(\sum_{j} \kappa_{j}^{m-1} u_{j j i}\right)^{2} }_{\operatorname{V}}.\nonumber
\end{align}

Without loss of generality, assume that $(-u) u_{11} \geq L$ for a sufficiently large constant $L\gg M>0$. Then, by applying \eqref{eq first derivative of test function} and \eqref{First derivation of the main eq}, we obtain
$$
\operatorname{I}=A f_{x_s} u_{s}+A f_{u}|D u|^{2}-M \frac{f_{p_s} u_{s}}{u}-B f_{p_s} x_{s} \geq-C_{A,B}+C\frac{M}{u}\geq -C_{A,B}-Cu_{11},
$$
where $C_{A,B}\leq C(A+B)$. Combining $\frac{\left(S_{k}\right)_{j}^{2}}{S_{k}} \leq C\left(1+u_{11}^{2}\right)$ with
\begin{align*}
    S_{k}^{p q ,r s} u_{p q j} u_{r s j} & =S_{k}^{p p, q q} u_{p p j} u_{q q j}-S_{k}^{p p, q q} u_{p q j}^{2} \\
    & =S_{k}^{p p ,q q} u_{p p j} u_{q q j}-2\sum_{p\ne j} S_{k}^{p p, j j} u_{p j j}^{2}-\sum_{p \neq j;q\ne j} S_{k}^{p p, q q} u_{p q j}^{2},
\end{align*}
it follows that
\begin{align*}
    \operatorname{II}\geq  & \frac{m}{P_m} \sum_j \kappa_j^{m-1} \z[-C(1+u_{11}^2)+2\frac{(S_k)_j^2}{S_k}-S_k^{pp,qq}u_{ppj}u_{qqj}+2S_k^{pp,jj}u_{pjj}^2\x] \\
    \geq& -Cm\left(\frac{1}{u_{11}}+u_{11}\right)+\frac{m}{P_{m}} \sum_{j} \kappa_{j}^{m-1}\left[2 \frac{(S_{k})_j^2}{S_k}-S_{k}^{pp,qq} u_{p p j} u_{q qj}\right] \\
    & +\frac{2 m}{P_{m}} \sum_{j} \kappa_{j}^{m-1} \sum_{p \neq j} S_{k}^{pp,jj} u_{p j j}^{2}.
\end{align*}
On the other hand,
\begin{align*}
    \operatorname{IV} & \geq \frac{2 m}{P_{m}} \sum_{i} \sum_{j \neq i} S_{k}^{i i} \frac{\kappa_{i}^{m-1} -\kappa_{j}^{m-1}}{\kappa_{i}-\kappa_{j}} u_{jii}^{2} = \frac{2 m}{P_{m}} \sum_{i} \sum_{j \ne i} S_{k}^{jj} \frac{\kappa_{j}^{m-1}- \kappa_{i}^{m-1}}{\kappa_{j}-\kappa_{i}} u_{jji}^{2}.
\end{align*}
Using \eqref{inequality 1} and the above estimation, we get
\begin{align}
    0  \geq &M \frac{S_{k}^{i i} u_{i i}}{u}-M S_{k}^{i i}\left(\frac{u_{i}}{u}\right)^{2}+A S_{k}^{i i} u_{i i}^{2}+B \sum_{i} S_{k}^{i i}-C_{A,B}-Cmu_{11}  \label{inequality 2} \\
    & +\frac{m}{P_{m}} \sum_{j} \kappa_{j}^{m-1}\left[2 \frac{(S_{k})_j^2}{S_k}-S_{k}^{pp,qq} u_{p p j} u_{q qj}\right] +\frac{2 m}{P_{m}} \sum_{j} \kappa_{j}^{m-1} \sum_{p \neq j} S_{k}^{pp,jj} u_{j jp }^{2} \nonumber \\
    & +\frac{2 m}{P_{m}} \sum_{i} \sum_{j \neq i} S_{k}^{jj} \frac{\kappa_{j}^{m-1}-\kappa_{i}^{m-1}}{\kappa_{j}-\kappa_{i}} u_{j j i}^{2}-\frac{m^{2}}{P_{m}^2} \sum_{i} S_{k}^{i i}\left(\sum_{j} \kappa_{j}^{m-1} u_{j j i}\right)^{2} \nonumber \\
    & +\frac{m(m-1)}{P_{m}} \sum_{i} \sum_{j} \kappa_{j}^{m-2} S_{k}^{i i} u_{j j i}^{2}.\nonumber
\end{align}

Let
\begin{align*}
    &A_{i}=\frac{m }{P_{m}} \kappa_i^{m-1} \left[2 \frac{(S_{k})_i^{2}}{S_{k}}-S_{k}^{p p ,q q} u_{p p i} u_{q i i}\right],&& B_{i}=\frac{2 m}{P_{m}} \sum_{j \neq i} \kappa_{j}^{m-1} S_{k}^{ii,jj} u_{jji}^{2}, \\
    &C_{i}=\frac{m(m-1)}{P_{m}} \sum_{j} S_{k}^{ii } \kappa_{j}^{m-2} u_{j j i}^{2},&& D_{i}=\frac{2 m}{P_{m}} \sum_{j \neq i} S_{k}^{j j} \frac{\kappa_{j}^{m-1}-\kappa_{i}^{m-1}}{\kappa_{j}-\kappa_{i}} u_{j j i}^{2}, \\
    & E_{i}=\frac{m^{2}}{P_{m}^{2}} S_{k}^{i i}\left(\sum_j \ka_{j}^{m-1} u_{j j i}\right)^{2}.&& \\
\end{align*}
Therefore, equation \eqref{inequality 2} becomes
\begin{align}\label{inequality 3}
    0  \geq  &\sum_{i}\left[A_{i}+B_{i}+C_{i}+D_{i}-E_{i}\right] \\
    & +M \frac{S_{k}^{i i} u_{i i}}{u}-M \frac{S_{k}^{i i} u_{i}^{2}}{u^{2}}+A S_{k}^{i i} u_{i i}^{2}+B \sum_{i} S_{k}^{i i}-C_{A,B}-Cmu_{11}.\nonumber
\end{align}

We first address the third-order derivative terms, and the proof is similar to that in \cite{LRW2}.
\begin{lemma}\label{lem A_i+... geq 0}
    For any $i\ne 1$, we have 
    $$
    A_{i}+B_{i}+C_{i}+D_{i}-\left(1+\frac{1}{m}\right) E_{i} \geq 0,\quad \te{for sufficiently large}\ m \ \te{and} \ \lam_1.
    $$
\end{lemma}
\begin{proof}
    By Lemma \ref{lem 2.3} (e), and letting $l=1$, we know for any $1\leq i \leq n$, $A_i\geq 0$. 
    
    Now compute the remaining terms. Let $\ga=1+\frac{1}{m}$,
    \begin{align*}
        & P_{m}^{2}\left[B_{i}+C_{i}+D_{i}-\ga E_{i}\right] \\
        = &\sum_{j \neq i} 2 m P_{m} \ka_{j}^{m-1} S_{k}^{i i, j j} u_{j j i}^{2}+m(m-1) P_{m} S_{k}^{i i} \sum_{j} \ka_{j}^{m-2} u_{j j i}^{2} \\
        &+ 2 m P_{m} \sum_{j \ne i} S_{k}^{jj} \frac{\ka_{j}^{m-1}-\ka_{i}^{m-1}}{\ka_{j}-\ka_{i}} u_{j j i}^{2} -m^{2} \ga S_{k}^{ii}\left(\sum_{j} \ka_{j}^{m-1} u_{j j i}\right)^{2} \\
        = & m P_{m} \sum_{j \neq i}\left(2 \ka_{j}^{m-1} S_{k}^{i i, j j}+(m-1) S_{k}^{i i} \ka_{j}^{m-2}+2 S_{k}^{j j} \frac{\ka_{j}^{m-1}-\ka_{i}^{m-1}}{\ka_{j}-\ka_{i}} \right)u_{jji}^2 \\
        &- m^{2} \ga \sum_{j\ne i} S_{k}^{i i} \ka_{j}^{2 m-2} u_{j j i}^{2} \\
        &+  m(m-1) P_{m} S_{k}^{i i} \ka_{i}^{m-2} u_{iii}^{2}-m^{2} \ga S_{k}^{i i} \ka_{i}^{2 m-2} u_{iii}^{2} (:=\lamm_2)\\
        &-  m^{2} \ga S_{k}^{i i} \sum_{p\ne i}\sum_{q \neq p, i} \ka_{p}^{m-1} u_{p p i} \ka_{q}^{m-1} u_{q q i} (:=\lamm_3)\\
        &-  2 m^{2} \ga S_{k}^{i i} \sum_{p \ne i} \ka_{p}^{m-1} u_{p p i} \ka_{i}^{m-1} u_{i i i}(:=\lamm_4) \\
        =& \lamm_1+\lamm_2+\lamm_3+\lamm_4.
    \end{align*}

     So let's first deal with 
     \begin{align*}
        \lamm_1 =&  m P_{m} \sum_{j \neq i}\left(2 \ka_{j}^{m-1} S_{k}^{i i, j j}+(m-1) S_{k}^{i i} \ka_{j}^{m-2}+2 S_{k}^{j j} \frac{\ka_{j}^{m-1}-\ka_{i}^{m-1}}{\ka_{j}-\ka_{i}} \right)u_{jji}^2 \\
        & - m^{2} \ga \sum_{j\ne i} S_{k}^{i i} \ka_{j}^{2 m-2} u_{j j i}^{2}. 
     \end{align*}
     For $j \neq i$, $(\lam_j-\lam_i)S_k^{ii,jj}=S_{k}^{ii}-S_k^{jj}$. Then 
    \begin{equation}\label{ka_i and ka_j}
        \ka_j S_k^{ii,jj}+S_k^{jj}=\ka_i S_k^{ii,jj}+S_k^{ii},
    \end{equation}
    and
    \begin{align*}
        & mP_{m}\left[2 \ka_{j}^{m-1} S_{k}^{i i,j j}+(m-1) S_{k}^{i i} \ka_{j}^{m-2}+2 S_{k}^{j j} \sum_{l=0}^{m-2} \ka_{i}^{m-2-l} \ka_{j}^{l}\right]-m^{2} \ga S_{k}^{i i} \ka_{j}^{2 m-2} \\
        & =m P_{m}\left[2 \ka_{j}^{m-2}\left(\ka_{j} S_{k}^{i i,j j}+S_{k}^{j j}\right)+(m-1) S_{k}^{ii} \ka_{j}^{m-2}+2 S_{k}^{j j} \sum_{l=0}^{m-3} \ka_{i}^{m-2-l} \ka_{j}^{l}\right]-m^{2} \ga S_{k}^{ii} \ka_{j}^{2 m-2} \\
        & = mP_{m}\left[2  \ka_{j}^{m-2}\left(\ka_{i} S_{k}^{i i,j j}+S_{k}^{i i}\right)+(m-1) S_{k}^{i i} \ka_{j}^{m-2}+2  S_{k}^{j j} \sum_{l=0}^{m-3} \ka_{i}^{m-2-l} \ka_{j}^{l}\right]-m^{2} \ga S_{k}^{i i} \ka_{j}^{2 m-2} \\
        & =m P_{m}\left[2 \ka_{j}^{m-2} \ka_{i} S_{k}^{i i , j j}+(m+1) S_{k}^{i i} \ka_{j}^{m-2}+2 S_{k}^{j j} \sum_{l=0}^{m-3}\ka_{i}^{m-2-l} \ka_{j}^{l}\right]-m^{2} \ga S_{k}^{i i} \ka_{j}^{2 m-2} \\
        & =mP_{m}\left[2  \ka_{j}^{m-2} \ka_{i} S_{k}^{ii, j j}+2  S_{k}^{j j} \sum_{l=0}^{m-3} \ka_{i}^{m-2-l} \ka_{j}^{l}\right]+m(m+1)  S_{k}^{i i} \ka_{j}^{m-2} \sum_{p \neq j} \ka_{p}^{m}.
    \end{align*}
    Thus
    \begin{align*}
        \lamm_1 =&  \sum_{j \ne i} mP_{m}\left[2  \ka_{j}^{m-2} \ka_{i} S_{k}^{ii, j j}+2  S_{k}^{j j} \sum_{l=0}^{m-3} \ka_{i}^{m-2-l} \ka_{j}^{l}\right]u_{jji}^2 \\
        & +  m(m+1)  S_{k}^{i i}\ka_{i}^{m} \sum_{j\ne i} \ka_{j}^{m-2}   u_{jji}^2 + m(m+1)  S_{k}^{i i} \sum_{j\ne i}\sum_{p\ne j, i} \ka_{j}^{m-2} \ka_{p}^{m}u_{jji}^2.
    \end{align*}

    For $\lamm_3$, since  
    \begin{align*}
        2 \sum_{p\ne i}\sum_{q \neq p,i}  \ka_{q}^{m-2} \ka_{p}^{m} u_{qq i}^{2}&=\sum_{p \neq i} \sum_{q \ne i, p} \ka_{p}^{m-2} \ka_{q}^{m} u_{p p i}^{2}+\sum_{p \neq i} \sum_{q \neq i, p} \ka_{q}^{m-2} \ka_{p}^{m} u_{q q i}^{2} \\
        &\geq  2 \sum_{p\ne i}\sum_{q \neq p ,i} \ka_{p}^{m-1} \ka_{q}^{m-1} u_{pp i}u_{qqi}.
    \end{align*}
    Hence
    \begin{align*}
        \lamm_3  \geq-m^{2} \ga S_{k}^{ii} \sum_{p\ne i}\sum_{q\ne p, i} \ka_{q}^{m-2} \ka_{p}^{m} u_{q q i}^{2}=-m(1+m) S_{k}^{i i}  \sum_{p\ne i}\sum_{q\ne p, i} \ka_{q}^{m-2} \ka_{p}^{m} u_{qq i}^{2}, 
    \end{align*}
    and
    \begin{align}\label{eq 1+3}
        \lamm_1 +\lamm_3  \geq&  \sum_{j \ne i} mP_{m}\left[2  \ka_{j}^{m-2} \ka_{i} S_{k}^{ii, j j}+2  S_{k}^{j j} \sum_{l=0}^{m-3} \ka_{i}^{m-2-l} \ka_{j}^{l}\right]u_{jji}^2 \\
        & +  m(m+1)  S_{k}^{i i}\ka_{i}^{m} \sum_{j\ne i} \ka_{j}^{m-2}   u_{jji}^2.\nonumber
    \end{align}
    
    Finally, 
    \begin{align}\label{eq 2}
        \lamm_2 %& =m(m-1) P_{m} S_{k}^{i i} \ka_{i}^{m-2} u_{i i i}^{2}-m^{2} \beta S_{k}^{i i} \ka_{i}^{2 m-2} u_{i i i}^{2}=m(m-1) S_{k}^{i i}\left(\sum_{j \neq i} \ka_{j}^{m}\right) \ka_{i}^{m-2} u_{iii}^{2} \\
        %& -2 m S_{k}^{11} \ka_{i}^{2 m-2} u_{iii}^{2}
        =m S_{k}^{i i}\left[(m-1) \sum_{j \neq i} \ka_{j}^{m}-2 \ka_{i}^{m}\right] \ka_{i}^{m-2} u_{iii}^{2}.
    \end{align}
    Combining \eqref{eq 1+3} with \eqref{eq 2} yields
        \begin{align*}
        \lamm_1 +\lamm_2+\lamm_3 +\lamm_4 
        \geq & 2 m \sum_{j\ne i} P_{m}\left[\ka_{j}^{m-2} \ka_{i} S_{k}^{i i, j j}+S_{k}^{j j} \sum_{l=0}^{m-3} \ka_{i}^{m-2-l} \ka_{j}^l\right] u_{j j i}^{2} \\
        &+ m(1+m) S_{k}^{i i} \sum_{j \neq i} \ka_{j}^{m-2} \ka_{i}^{m} u_{j j i}^{2} \\
        &+ m S_{k}^{i i}\left[(m-1) \sum_{j \neq i} \ka_{j}^{m}-2 \ka_{i}^{m}\right] \ka_{i}^{m-2} u_{i i i}^{2} \\
        &- 2 m^{2} \ga S_{k}^{i i} \sum_{p \neq i} u_{p p i} \ka_{p}^{m-1} \ka_{i}^{m-1} u_{i i i} \\
        \triangleq&  \mathcal{P},
    \end{align*}
    
    Now, let us consider the term $\ka_{j}^{m-2} \ka_{i} S_{k}^{i i, jj}+S_k^{jj}\sum_{l=0}^{m-3} \ka_{i}^{m-2-l} \ka_{j}^{l}$. Using \eqref{ka_i and ka_j} multiple times, we have 
    \begin{eqnarray*}
        && \ka_{j}^{m-2} \ka_{i} S_{k}^{i i, jj}+S_k^{jj}\sum_{l=0}^{m-3} \ka_{i}^{m-2-l} \ka_{j}^{l} \\
        &=& \ka_{j}^{m-3} \ka_{i}\z[\ka_j S_k^{i i, jj}+S_{k}^{jj}\x]+S_{k}^{jj}\sum_{l=0}^{m-4} \ka_{i}^{m-2-l} \ka_{j}^{l}  \\
        &=& \ka_{j}^{m-3} \ka_{i}\left[\ka_i S_k^{i i, jj}+S_{k}^{i i}\right]+S_{k}^{jj}\sum_{l=0}^{m-4} \ka_{i}^{m-2-l} \ka_{j}^{l}  \\
        &=& \ka_{j}^{m-3} \ka_{i} S_{k}^{i i}+\z[\ka_{j}^{m-3} \ka_{i}^{2} S_{k}^{i i, jj} +S_{k}^{jj}\sum_{l=0}^{m-4} \ka_{i}^{m-2-l} \ka_{j}^{l}\x]  \\
        &=&  \cdots \\
        &=& \left[\ka_{j}^{m-3} \ka_{i}+\ka_{j}^{m-4} \ka_{i}^{2}+\cdots+\ka_{j} \ka_{i}^{m-3}+\ka_{i}^{m-2}\right] S_{k}^{ii} +\ka_i^{m-1}S_k^{ii,jj}. 
    \end{eqnarray*}
    Hence
    $$ 2 m P_{m}\left[\ka_{j}^{m-2} \ka_{i} S_{k}^{ii, jj}+S_{k}^{jj} \sum_{l=0}^{m-3} \ka_{i}^{m-2-l} \ka_{j}^{l}\right] u_{jj i}^{2} \geq 2 m(m-2) S_{k}^{i i} \ka_{j}^{m-2}\ka_i^{m}  u_{jji}^{2},$$
    and
    \begin{align*}
        \mathcal{P}
        \geq m S_k^{ii} \sum_{j\ne i} &\bigg[ (3m-3)\ka_j^{m-2}\ka_i^m u_{jji}^2 -2(1+m) \ka_{j}^{m-1}\ka_i^{m-1}u_{jji}u_{iii} \\
        & +(m-3)\ka_j^{m}\ka_i^{m-2}u_{iii}^2\bigg]\geq 0.
    \end{align*}
    Here, we choose $m$ sufficiently large such that $(3m-3)(m-3)\geq (1+m)^2$.

    Therefore, by choosing $m$ sufficiently large, we have
    $$
        P_m^2\z[B_i+C_i+D_i-(1+\frac{1}{m})E_i\x]\geq \mathcal{P} \geq 0,\quad \te{for any } i \ne 1.
    $$
\end{proof}

Next, for $i=1$, we prove the following lemma.
\begin{lemma}\label{lem A_1+... geq ...}
    For $\mu=1,\cdots,k-1$, if there exists some positive constant $\de_\mu \leq 1$, such that $\frac{\lam_{\mu}}{\lam_1}\geq \de_{\mu} $. Then there exist two sufficiently small positive constants $\e$, $\de_{\mu+1}$ depending on $\de_{\mu}$, such that, if $\frac{\lam_{\mu+1}}{\lam_1}\leq \de_{\mu+1}$ and $\lam_1$ is sufficiently large, we have
    \begin{align}\label{eq A_1+B_1+... 0} 
        \frac{P_m^2}{m}\z[A_1+B_1+C_1+D_1-\z(1+\frac{\e}{m}\x)E_1\x] \geq& (1+\e) \z[S_k-\ka_1 S_k^{11}\x] \ka_1^{2m-3} u_{111}^2 \\
        & +2 P_m\sum_{j \ne 1}S_k^{jj}\sum_{l=1}^{m-4}\ka_j^l \ka_1^{m-2-l} u_{jj1}^2.\nonumber 
    \end{align}
\end{lemma}
\begin{proof}
    Computing similar to \cite[(3.18)]{LRW2}, we obtain
    \begin{align*}
        & \frac{P_{m}^{2}}{m}\left[B_{1}+C_{1}+D_{1}-\z(1+\frac{\e}{m}\x)E_{1}\right] \\
        \geq & \sum_{j \neq 1}\left[(1-\e)P_{m}+(m+\e) \ka_{1}^{m}\right] \ka_{j}^{m-2} S_{k}^{11} u_{j j 1}^{2} \\
        &+  \left[(m-1)\sum_{j\ne 1}\ka_j^m-(1+\e)\ka_{1}^{m}\right] \ka_{1}^{m-2} S_{k}^{11} u_{111}^{2} \\
        &-2(m+\e)  S_{k}^{11} \sum_{j \neq 1} u_{j j 1} \ka_{j}^{m-1}  \ka_{1}^{m-1} u_{111} +2 P_{m} \sum_{j \neq 1} S_{k}^{j j} \sum_{l=0}^{m-3} \ka_{1}^{m-2-l} \ka_{j}^{l} u_{j j 1}^{2}.
    \end{align*}

    Since $S_{k}^{j j} \geq S_k^{11}$ for $j \neq 1$, it follows that
    \begin{align*}
        2P_m\sum_{j\ne 1}S_k^{jj} \sum_{l=0}^{m-3}\ka_1^{m-2-l}\ka_j^l u_{jj1}^2 
        \geq & 2P_m \ka_1^{m-2}\sum_{j\ne 1}S_k^{jj}u_{jj1}^2+2\ka_1^{m+1} \sum_{j\ne 1}S_k^{11}\ka_j^{m-3} u_{jj1}^2 \\
        & + 2P_m \sum_{j\ne 1}S_k^{jj}\sum_{l=1}^{m-4}\ka_1^{m-2-l}\ka_j^l u_{jj1}^2.
    \end{align*}
    Then
    \begin{align*}
        & \frac{P_{m}^{2}}{m}\left[B_{1}+C_{1}+D_{1}-\z(1+\frac{\e}{m}\x)E_{1}\right] \\
        \geq&  S_k^{11} \sum_{j \neq 1} \z[ (m+3)  \ka_{1}^{m} \ka_{j}^{m-2}  u_{j j 1}^{2}-2 (m+\e) \ka_{j}^{m-1} \ka_{1}^{m-1}  u_{j j1} u_{111}  +(m-1) \ka_{1}^{m-2}\ka_j^{m}  u_{111}^{2} \x] \\
        &- (1+\e)\ka_{1}^{2 m-2} S_{k}^{11} u_{111}^{2}+2 P_{m} \ka_{1}^{m-2} \sum_{j \neq 1} S_{k}^{j j} u_{j j 1}^{2} \\
        &+ 2 P_{m} \sum_{j\ne 1} S_k^{j j} \sum_{l=1}^{m-4} \ka_{1}^{m-2-l} \ka_{j}^{l} u_{j j 1}^{2} \\
        \geq&  -(1+\e)\ka_{1}^{2 m-2} S_{k}^{11} u_{111}^{2}+2 P_{m} \ka_{1}^{m-2} \sum_{j\ne 1} S_{k}^{j j} u_{j j 1}^{2} \\
        &+  2 P_{m} \sum_{j \neq 1} S_{k}^{j j} \sum_{l=1}^{m-4} \ka_{1}^{m-2-l} \ka_{j}^{l} u_{j j 1}^{2}.
    \end{align*}
    Here, we have used the inequality $(m+3)(m-1)-(m+\e)^{2}>0$ for large $m$. This can be rewritten as
    \begin{align}\label{eq B_1+...}
        &\frac{P_{m}^{2}}{m}\left[B_{1}+C_1+D_{1}-\z(1+\frac{\e}{m}\x)E_{1}\right] \\
        \geq& -(1+\e)\ka_{1}^{2 m-3} S_{k} u_{111}^{2}+(1+\e)\left(S_{k}-\ka_{1} S_{k}^{11}\right) \ka_{1}^{2 m-3} u_{111}^{2} \nonumber \\
        &+ 2 P_{m} \ka_{1}^{m-2} \sum_{j \ne 1} S_{k}^{j j} u_{j j1}^{2}+2 P_{m} \sum_{j\ne 1} S_{k}^{j j} \sum_{l = 1}^{m-4} \ka_{1}^{m-2-l} \ka_{j}^{l} u_{j j 1}^{2}.\nonumber
    \end{align}

    Next, we  split the term $A_1=m\frac{ \ka_{1}^{m-1}}{P_{m}}\left[2 \frac{(S_k)_{1}^{2}}{S_{k}}-S_{k}^{pp,qq} u_{pq1} u_{qq1}\right]$ into three cases and, ultimately, complete the proof of this lemma.
    
    \textbf{Case 1.} For $\mu=1$, $\de_1=1$, and $\lam_2 \leq \de_2 \lam_1$, where $\de_2$ is a small parameter to be determined later. 
    By Lemma \ref{lem 2.3} (e), and letting $l=1$ and $\tau=\frac{1}{k-1}$, we have
    \begin{align}\label{eq A_1 2}
        \frac{P_{m}^{2}}{m} A_{1} & \geq P_{m} \ka_{1}^{m-1} \z(1+\tau(1-\delta)\x) \frac{S_k }{S_{1}^{2}} \left(\sum_j u_{jj1}\right)^2 \\
        % & \geq P_{m} \ka_{1}^{m-3} S_{k} \z(1+\tau(1-\delta)\x) {\left(1+(n-2)\delta_{2}\right)^{2}}\left(\sum_{j} u_{j j 1}\right)^{2} \nonumber \\
        & \geq P_{m} \ka_{1}^{m-3} S_{k}\z(1+\frac{\tau}{2}(1-\delta)\x) \left(\sum_j  u_{j j 1}\right)^{2} \nonumber\\
        & \geq P_m \ka_1^{m-3}S_k \z[  \z(1+\frac{\tau}{4}(1-\delta)\x) u_{111}^{2}-C \sum_{j\ne 1} u_{jj 1}^2 \x], \nonumber
    \end{align}
    for sufficiently small $\de$, $\de_2$ and $\frac{1}{\lam_1}$. Combining \eqref{eq B_1+...} with \eqref{eq A_1 2} yields
    \begin{align*}
        & \frac{P_{m}^{2}}{m}\left[A_{1}+B_{1}+C_{1}+D_{1}-\z(1+\frac{\e}{m}\x)E_1\right] \\
        \geq & \z[\frac{\tau}{4}(1-\delta)-\e\x]S_{k} \ka_{1}^{2 m-3}  u_{111}^{2} + (1+\e)\left(S_{k}-\ka_{1} S_k^{11}\right) \ka_{1}^{2 m-3} u_{111}^{2} \\
        &+  P_{m} \ka_{1}^{m-3} \sum_{j\ne 1}  \z[2\ka_1 S_{k}^{jj}-C S_k\x] u_{j j 1}^{2} +2 P_m \sum_{j\ne 1} S_{k}^{jj} \sum_{l=1}^{m-4} \ka_{1}^{m-2-l} \ka_{j}^l  u_{j j 1}^{2}.
    \end{align*}
    It follows form Lemma \ref{lem 2.3} (b) that $\ka_1 S_{k}^{j j}  \geq \frac{\lam_2}{\delta_{2}} S_{k}^{22} \geq \frac{\ttt}{\delta_{2}} S_{k}$, for any $j\ne 1$.
    %$$ P_{m } \ka_{1}^{m-3}\sum_{j \neq 1}\left[\ka_{1} S_{k}^{jj}-C S_{k}\right] u_{j j 1}^{2} \geq P_{m} \ka_{1}^{m-3} \sum_{j \neq 1}\left[\frac{C_{1}}{\delta_{2}}-C\right] S_{k} u_{j j1}^{2} \geq 0, \quad \te{for } \de_2 \te{ small}. $$
    Therefore, for a fixed small $\de$, we choose $\de_2$, $\e$, and $\frac{1}{\lam_1}$ sufficient small, such that
    \begin{align}
        \frac{P_{m}^{2}}{m}\left[A_{1}+B_{1}+C_{1}+D_{1}-\z(1+\frac{\e}{m}\x)E_{1}\right]  
        \geq&  (1+\e)\left(S_{k}-\ka_{1} S_{k}^{11}\right) \ka_{1}^{2 m-3} u_{111}^{2} \label{eq A_1+B_1+... 1} \\
        & %+P_{m} \ka_{1}^{m-2} \sum_{j\ne 1} S_{k}^{j j} u_{j j 1}^{2} 
        +2 P_{m} \sum_{j\ne 1} S_{k}^{jj} \sum_{l=1}^{m-4} \ka_{1}^{m-2-l} \ka_{j}^{l} u_{j j 1}^{2}. \nonumber
    \end{align}

    \textbf{Case 2.} For $2\leq \mu \leq k-2$, that is, $\lam_{\mu}\geq \de_\mu \lam_1$ and $\lam_{\mu+1}\leq \de_{\mu+1}\lam_1$, where $\de_{\mu+1}$ is a small parameter to be determined later. Applying Lemma \ref{lem 2.3} (e) again and letting $\tau=\frac{1}{k-\mu}$, we obtain
    \begin{align}\label{eq A_1 3}
        \frac{P_m^2}{m}A_1 \geq &(1+\tau(1-\de)) P_m \ka_1^{m-1}  \frac{S_k}{S_{\mu}^2}\z(\sum_j S_\mu^{jj}u_{jj1}\x)^2 -P_m\ka_1^{m-1}\frac{S_k}{S_\mu}S_{\mu}^{pp,qq}u_{pp1}u_{qq1}\nonumber\\
        \geq &   P_m \ka_1^{m-1} \frac{S_k}{S_\mu^2} \z[\sum_j (S_\mu^{jj}u_{jj1})^2 +\sum_{p\ne q}S_\mu^{pp}S_\mu^{qq}u_{pp1}u_{qq1}\x] \nonumber\\
        & - P_m \ka_1^{m-1} \frac{S_k}{S_\mu}S_{\mu}^{pp,qq}u_{pp1}u_{qq1} \nonumber\\
        =&   P_m \ka_1^{m-1} \frac{S_k}{S_\mu^2} \sum_j (S_\mu^{jj}u_{jj1})^2 + P_m \ka_1^{m-1} \frac{S_k}{S_\mu^2}\sum_{p\ne q} \z[S_{\mu}^{pp}S_\mu^{qq}-S_\mu S_\mu^{pp,qq}\x]u_{pp1}u_{qq1}.
    \end{align}

    Next, we split $\sum_{p\ne q}\z(S_{\mu}^{p p} S_{\mu}^{q q}-S_{\mu} S_{\mu}^{pp, q q}\x)u_{pp1}u_{qq1}$ into 
    three terms to deal with: 
    $$
        \sum_{p\ne q} \cdots=\sum_{p\ne q;p,q\leq \mu}\cdots +2\sum_{p\leq \mu;q>\mu}\cdots +\sum_{p\ne q; p,q>\mu}\cdots :=T_1+T_2+T_3.
    $$
    By Lemma \ref{lem 2.3} (c), for $p\ne q$, we have
    $$
    S_{\mu}^{p p} S_{\mu}^{q q}-S_{\mu} S_{\mu}^{pp, q q}=S_{\mu-1}^{2}(\lambda |p q)-S_{\mu}(\lambda |p q) S_{\mu-2}(\lambda | p q)\geq 0.
    $$

    First, let us consider $T_1$. We claim that for $p,q \leq \mu$, and $\de_{\mu+1},  \frac{1}{\lam_1}$ are sufficiently small,
    $$
        |S_{\mu}^{p p} S_{\mu}^{q q}-S_{\mu} S_{\mu}^{pp, q q}|\leq \es S_{\mu}^{p p} S_{\mu}^{q q},
    $$
    where small constant $\es>0$ depends on $\de_{\mu+1}$ and $\lam_1$. 
    
    Since $ \lambda \in \gaa_{k-1} \subset \gaa_{\mu+1}$, then
    \begin{align*}
        & S_{\mu}^{pp}=S_{\mu-1}(\lambda| p) \geq \frac{\lambda_{1 } \cdots\lambda_{\mu}}{\lambda_{p}}, \\
        & 0< S_{\mu-1}(\lam| pq) \leq C_{n-2}^{\mu-1} \frac{\lambda_{1} \cdots \lambda_{\mu+1}}{\lambda_{p} \lambda_{q}}, \\
        & 0< S_{\mu-2}(\lam| pq) \leq C_{n-2}^{\mu-2} \frac{\lambda_{1} \cdots \lambda_{\mu}}{\lambda_{p} \lambda_{q}}.
    \end{align*}
    On the other hand,
    \begin{align*}
        \left|S_{\mu}(\lambda | p q)\right| & \leq \frac{\lambda_{1} \cdots \lambda_{\mu}}{\lambda_{p} \lambda_{q}} \z[C\lam_{\mu+1}|\lam_{\mu+2}|+CK_0^2\x] \leq C\frac{\lambda_{1} \cdots \lambda_{\mu}}{\lambda_{p} \lambda_{q}} \z[\lam_{\mu+1}^2+K_0^2\x].
    \end{align*}
    Hence
    $$
    \z|S_{\mu-1}^2(\lam|pq)- S_{\mu}(\lambda | p q) S_{\mu-2}(\lambda | p q) \x| \leq C \z[\z(\frac{\de_{\mu+1}}{\de_\mu}\x)^2+\frac{1}{\de_\mu^2 \lam_1^2}\x]S_\mu^{pp}S_\mu^{qq},\quad \te{for } p,q \leq \mu.
    $$
    and 
    \begin{equation}\label{T_1}
        T_1\geq -\es \sum_{p\leq \mu}(S_\mu^{pp}u_{pp1})^2, \quad \te{provided that } \de_{\mu+1}, \ \frac{1}{\lam_1} \te{ are sufficiently small}.
    \end{equation}

    For the rest of the case. Since $0< S_{\mu}^{p p} S_{\mu}^{q q}-S_{\mu} S_{\mu}^{p p, q q} \leq S_{\mu}^{p p} S_{\mu}^{q q}$, 
    \begin{align}
        & T_2 \geq -2 \sum_{p\leq \mu;q>\mu} S_{\mu}^{p p}S_{\mu}^{qq}|u_{pp1}u_{qq1}| \geq -\es \sum_{p\leq \mu;q>\mu} (S_{\mu}^{p p}u_{pp1})^2-\frac{1}{\es}\sum_{p\leq \mu;q>\mu} (S_{\mu}^{qq}u_{qq1})^2,\label{T_2}\\
        & T_3\geq -\sum_{p\ne q;p,q>\mu} S_{\mu}^{p p}S_{\mu}^{qq}|u_{pp1}u_{qq1}|\geq -\sum_{p\ne q;p,q>\mu} (S_{\mu}^{p p}u_{pp1})^2.\label{T_3}
    \end{align}
    Substituting \eqref{T_1}, \eqref{T_2} and \eqref{T_3} into \eqref{eq A_1 3}, we obtain
    \begin{align}\label{eq A_1 4}
        \frac{P_m^2}{m}A_1 &\geq (1-\es)P_m \ka_1^{m-1} \frac{S_k}{S_\mu^2} \sum_{p\leq \mu}(S_\mu^{pp}u_{pp1})^2-C_\es P_m \ka_1^{m-1} \frac{S_k}{S_\mu^2} \sum_{q> \mu}(S_\mu^{qq}u_{qq1})^2 \nonumber \\
        & \geq (1-\es)(1+\de_\mu^m)\z(1-C\frac{\de_{\mu+1}}{\de_\mu}\x)\ka_1^{2m-3}S_k\sum_{p\leq \mu}u_{pp1}^2 -C_\es  P_m \ka_1^{m-3} \frac{S_k}{\de_\mu^2} \sum_{q> \mu}u_{qq1}^2.
    \end{align}
    Here, we use the fact that 
    \begin{equation}\label{2 leq mu leq k-2}
        \begin{aligned}
            & \frac{\ka_p S_\mu^{pp}}{S_\mu}\geq 1-\frac{S_\mu(\lam|p)}{S_\mu}
            %\geq 1-C_{n-1}^{\mu}\frac{\lam_{\mu+1}}{\lam_p}
            \geq 1-C\frac{\de_{\mu+1}}{\de_\mu}, \quad \te{for}\  p\leq \mu,\\
            & \frac{P_m}{\ka_1^m}\geq 1+\de_\mu^m, \quad \te{and} \ \z|\frac{\ka_1S_\mu^{qq}}{S_\mu}\x|\leq C\ka_1 \frac{\lam_1 \cdots \lam_{\mu-1}}{\lam_1 \cdots \lam_\mu}\leq \frac{C}{\de_\mu}, \quad \te{for} \ q>\mu.
        \end{aligned}
    \end{equation}
    
    Therefore, Combining \eqref{eq B_1+...} with \eqref{eq A_1 4} and letting $\es$, $\e$ and $\de_{\mu+1}$ be sufficiently small, we have
    \begin{align*}
        \frac{P_{m}^{2}}{m}& \left[A_{1}+B_{1}+C_{1}+D_{1}-\z(1+\frac{\e}{m}\x)E_{1}\right] \\
        \geq&  \left(1+\frac{\de_\mu^{m}}{2}\right) \ka_{1}^{2 m-3}S_k \sum_{p\leq \mu} u_{p p 1}^{2}-C P_{m} \ka_{1}^{m-3}\frac{ S_{k}}{\de_{\mu}^{2}} \sum_{q>\mu} u_{qq1}^{2} \\
        & -(1+\e)\ka_{1}^{2 m-3} S_{k}u_{111}^{2} +2 P_{m} \ka_{1}^{m-2} \sum_{j\ne 1} S_{k}^{j j} u_{j j 1}^{2} \\
        & +(1+\e)\left(S_{k}-\ka_{1} S_{k}^{11}\right) \ka_{1}^{2 m-3} u_{111}^{2}+2 P_{m} \sum_{j\ne 1} S_{k}^{j j} \sum_{l=1}^{m-4} \ka_{1}^{m-2-l} \ka_{j}^{l} u_{j j 1}^{2} \\
        \geq&  (1+\e)\left(S_{k}-\ka_{1} S_{k}^{11}\right) \ka_{1}^{2 m-3} u_{111}^{2} +\sum_{q >\mu} P_{m} \ka_{1}^{m-3}\left[2 \ka_{1} S_{k}^{q q}-C \frac{S_{k}}{\de_\mu^2}\right] u_{q q 1}^{2} \\
        % & + \z(1+\frac{\de_\mu^m}{2}\x)\ka_1^{2m-3}S_k \sum_{2\leq p\leq \mu} u_{pp1}^2 +2P_m\ka_1^{m-2}\sum_{2\leq p\leq \mu}S_k^{pp}u_{pp1}^2 \\
        & + 2 P_{m} \sum_{j\ne 1} S_{k}^{jj} \sum_{k=1}^{m-4} \ka_{1}^{m-2-l} \ka_{j}^{l} u_{j j1}^{2}.
    \end{align*}
    Using Lemma \ref{lem 2.3} (b), for any $q>\mu$, $\ka_{1} S_{k}^{q q}  \geq \frac{\lam_{\mu+1}}{\de_{\mu+1}} S_{k}^{\mu+1 \ \mu+1} \geq \frac{\ttt S_{k}}{\delta_{\mu+1}}$. If $\e$ and $\de_{\mu+1}$ are sufficiently small, we have
    \begin{align}
        \frac{P_{m}^{2}}{m}\left[A_{1}+B_{1}+C_{1}+D_{1}-\z(1+\frac{\e}{m}\x)E_{1}\right]  \geq &(1+\e)\left(S_{k}-\ka_{1} S_{k}^{11}\right) \ka_{1}^{2 m -3} u_{111}^{2} \label{eq A_1+B_1+... 2} \\
        & +2 P_{m} \sum_{j \neq 1} S_{k}^{jj} \sum_{l=1}^{m-4} \ka_{1}^{m - 2-l} \ka_{j}^{l} u_{j j 1}^{2}. \nonumber 
    \end{align}

    \textbf{Case 3.} For $\mu=k-1$, that is, $\lambda_{k-1} \geq \delta_{k-1} \lambda_{1}$. By Lemma \ref{lem 2.5} (b), $|\lam_k|\leq c_0 K_0$. From the discussion in Case 2,
    \begin{align*}
        & \frac{P_{m}^{2}}{m} A_{1} \geq P_{m}\ka_{1}^{m-1} \frac{S_k}{S_{k-1}^2}\sum_{j} \z(S_{k-1}^{j j} u_{j j 1}\x)^{2}+P_{m} \ka_{1}^{m-1} \frac{S_{k}}{S_{k-1}^{2}}\left[S_{k-1}^{p p} S_{k-1}^{qq}-S_{k-1} S_{k-1}^{pp, q q}\right] u_{p p 1} u_{q q 1},
    \end{align*}
    and 
    $$
    S_{k-1}^{p p} S_{k-1}^{q q}-S_{k-1} S_{k-1}^{p p, q q}=S_{k-2}^{2}(\lam|pq)-S_{k-1}(\lam|pq) S_{k-3}(\lam|pq).
    $$
    
    We also claim that for any $p, q \leq \mu=k-1$, and $\lam_1$ is sufficiently large,
    $$
    |S_{k-2}^{2}(\lam|pq)-S_{k-1}(\lam|pq) S_{k-3}(\lam|pq)|\leq \es S_{k-1}^{p p} S_{k-1}^{q q},
    $$ 
    where small constant $\es$ depends on $\lam_1$.
    
    Since $\lam \in \gaa_{k-1}$, it follows that 
    $$
    S_{k-1}^{p p}\geq \sigma_{k-2}(\lam|p)\geq \frac{\lambda_{1} \cdots \lambda_{k-1}-C\lam_1\cdots \lam_{k-2}\z[|\lam_k|+K_0\x]}{\lambda_{p}}\geq \frac{\lambda_{1} \cdots \lambda_{k-1}}{2\lambda_{p}}.
    $$
    On the other hand, 
    \begin{align*}
        & |S_{k-1}(\lam|pq)|=|\si_{k-1}(\lam|pq)+\al \si_{k-2}(\lam|pq)|\leq C\frac{\lam_1 \cdots \lam_{k-1}}{\lam_p \lam_q} \z[\lam_k^2+K_0^2\x],\\
        & |S_{k-2}(\lam|pq)|\leq C\frac{\lam_1 \cdots \lam_{k-1}}{\lam_p \lam_q} \z[|\lam_k|+K_0\x], \quad  |S_{k-3}(\lam|pq)|\leq C\frac{\lam_1 \cdots \lam_{k-1}}{\lam_p \lam_q}.
    \end{align*}
    Hence 
    $$
    |S_{k-2}^{2}(\lam|pq)-S_{k-1}(\lam|pq) S_{k-3}(\lam|pq)| \leq \frac{CK_0}{\de_{k-1}^2\lam_1^2} S_{k-1}^{p p}S_{k-1}^{qq}.
    $$
    Similar to the analysis in Case 2, we obtain $\frac{P_m}{\ka_1^m}\geq 1+\de_{k-1}^m$ and for any $p\leq k-1$,
    \begin{align*}
        \frac{\ka_{p} S_{k-1}^{p p}}{S_{k-1}} 
        %=1-\frac{S_{k-1}(\lam|p)}{S_{k-1}} 
        \geq 1-C\frac{\lam_1 \cdots \lam_{k-1}[|\lam_k|+K_0]/\lam_p}{\lam_1 \cdots \lam_{k-1}} \geq 1-\frac{CK_0}{\de_{k-1}\lam_1}.
    \end{align*}
    Furthermore, for any $q>k-1$, 
    $$
    \z|\frac{\ka_{1} S_{k-1}^{q q}}{S_{k-1}} \x| \leq \ka_1\frac{C \lam_1\cdots \lam_{k-2}}{c\lam_1 \cdots \lam_{k-1}}\leq \frac{C}{\de_{k-1}}.
    $$ 
    Here, we use the fact that $|\lam_k|\leq CK_0$. 
    Summarily, if $\lam_1$ is sufficiently large, we have
    \begin{equation*}
        \frac{P_{m}^2A_1}{m} \geq \z(1+\frac{\de_{k-1}^m}{2}\x) \ka_1^{2m-3}S_k \sum_{p\leq k-1} u_{pp1}^2-CP_m\ka_1^{m-3}\frac{S_k}{\de_{k-1}^2}\sum_{q>k-1} u_{qq1}^2.
    \end{equation*}
    Substituting this into \eqref{eq B_1+...} and letting $\e$ be small yields
    \begin{align*}
        \frac{P_{m}^{2}}{m}\left[A_{1}+B_{1}+C_{1}+D_{1}-\z(1+\frac{\e}{m}\x)E_{1}\right]  \geq &(1+\e) \left[S_{k}-\ka_{1} S_{k}^{11}\right] \ka_{1}^{2 m-3} u_{111}^{2}  \\
        & +P_{m} \ka_{1}^{m-3} \sum_{q > k-1}\left[2 \ka_{1} S_{k}^{q q}- \frac{C}{\delta_{k-1}} S_{k}\right] u_{q q 1}^{2} \\
        & +2 P_{m} \sum_{j\ne 1} S_k^{jj} \sum_{l=1}^{m-4} \ka_{1}^{m-2-l} \ka_{j}^l u_{j j 1}^{2}.
    \end{align*}

    In the end, we prove that 
    $$\frac{\ka_{1} S_{k}^{qq}}{S_{k}} \gg 1,\quad \te{for any } q>k-1.
    $$
    First, it is important to note that if $\lam_1$ is sufficiently large, then
    $$
    S_{k}^{q q}\geq \si_{k-1}(\lam|q)\geq \lambda_{1} \cdots \lambda_{k-1}-C \lambda_{1} \cdots \lambda_{k-2}\left[|\lambda_{k}|+K_{0}\right] \geq \frac{\lambda_{1} \cdots \lambda_{k-1}}{2},
    $$
    and
    $
    0<c\leq S_{k} \leq C \lambda_{1} \cdots \lambda_{k-1}\left[|\lambda_{k}| +K_{0}\right]\leq C \lambda_{1} \cdots \lambda_{k-1} K_0.
    $ Hence
    $$
    \frac{\ka_{1} S_{k}^{qq}}{S_{k}} \geq C\lam_1  \gg 1, \quad \text{if } \lam_1 \text{ is large}.
    $$
    Therefore, we have
    \begin{align}
        \frac{P_{m}^{2}}{m}\left[A_{1}+B_{1}+C_{1}+D_{1}-\z(1+\frac{\e}{m}\x)E_{1}\right]\geq  & (1+\e)\left[S_{k}-\ka_{1} S_{k}^{11}\right] \ka_{1}^{2 m-3} u_{111}^{2} \label{eq A_1+B_1+... 3}\\
        & +2 P_{m} \sum_{j\ne 1} S_{k}^{j j} \sum_{l=1}^{m-4} \ka_{1}^{m-2-l} \ka_{j}^l  u_{j j 1}^{2}.\nonumber
    \end{align}
    By combining equations \eqref{eq A_1+B_1+... 1}, \eqref{eq A_1+B_1+... 2} and \eqref{eq A_1+B_1+... 3}, we complete the proof.
\end{proof}

In the next lemma,  we focus on the term 
$$\left(S_{k}-\ka_{1} S_{k}^{11}\right) \ka_{1}^{2 m-3} u_{111}^{2}+2 P_{m} \sum_{j\ne 1} S_{k}^{j j} \sum_{l=1}^{m-4} \ka_{1}^{m-2-l} \ka_{j}^l  u_{j j 1}^{2}.$$ 
\begin{lemma}\label{lem 3.3}
    For sufficiently large $m$, we have
    \begin{align}
        (1+\e)&\left(S_{k}-\ka_{1} S_{k}^{11}\right) \ka_{1}^{2 m-3} u_{111}^{2} +2 P_{m} \sum_{j\ne 1} S_{k}^{j j} \sum_{l=1}^{m-4} \ka_{1}^{m-2-l} \ka_{j}^l  u_{j j 1}^{2} \\
        & \geq \min\z\{0,C \z[\frac{M^2}{m^2 L^2}+\frac{A^2}{m^2}+\frac{B^2}{m^2 \lam_1^2}\x](S_{k}-\ka_{1} S_{k}^{11})  P_{m}^2 \lam_1\x\}, \nonumber
    \end{align}
    Here, $\e$ is the small constant defined in Lemma $\ref{lem A_1+... geq ...}$.
    As a consequence, \eqref{eq A_1+B_1+... 0} becomes
    \begin{align}\label{eq A_1+B_1+... 4}
        A_1+B_1& +C_1+D_1-\z(1+\frac{\e}{m}\x)E_1 \\
        & \geq \min\z\{0,C \z[\frac{M^2}{m L^2}+\frac{A^2}{m}+\frac{B^2}{m \lam_1^2}\x](S_{k}-\ka_{1} S_{k}^{11})  \lam_1\x\}.\nonumber 
    \end{align}
\end{lemma}
\begin{proof}
    Suppose $S_{k}-\ka_{1} S_k^{11} \geq 0$, in this case, the proof is complete. We now consider the case when $S_{k}-\ka_{1} S_k^{11} < 0$. Applying \eqref{eq first derivative of test function} and the Cauchy inequality, we obtain
    %$$-\ka_{1}^{m-1} u_{111}=\sum_{j \neq 1} \ka_{j}^{m-1} u_{j j 1}+P_{m}\left[\frac{M}{m} \frac{u_{1}}{u}+\frac{A}{m} u_{1} u_{11}+\frac{B}{m} x_{1}\right],$$
    \begin{align*}
        \ka_{1}^{2 m-2} u_{111}^{2} & \leq(n+2)\left[\sum_{j\ne 1} \ka_{j}^{2 m-2} u_{j j 1}^{2}+P_{m}^{2} \frac{M^{2}}{m^{2}} \frac{u_{1}^{2}}{u^{2}}+P_{m}^{2} \frac{A^{2}}{m^{2}} u_{1}^{2} u_{11}^{2}+P_{m}^{2} \frac{B^{2}}{m^{2}} x_{1}^{2}\right] \\
        & \leq(n+2) \sum_{j\ne 1} \ka_{j}^{2m-2} u_{j j 1}^{2}+C  \z[\frac{M^2}{m^2 L^2}+\frac{A^2}{m^2}+\frac{B^2}{m^2 u_{11}^2}\x] P_{m}^{2}u_{11}^2.
    \end{align*}
    Multiplying it by $\frac{S_{k}-\ka_{1} S_{k}^{11}}{\ka_{1}}$, we obtain
    \begin{align*}
        (S_k-\ka_1S_k^{11})\ka_1^{2m-3}u_{111}^2  \geq& -(n+2)  S_k^{11} \sum_{j\ne 1}\ka_j^{2m-2}u_{jj1}^2 \\
        & +C\z[\frac{M^2}{m^2 L^2}+\frac{A^2}{m^2}+\frac{B^2}{m^2 u_{11}^2}\x]\frac{S_k-\ka_1 S_k^{11}}{\ka_1} P_m^2u_{11}^2.
    \end{align*}
    Therefore
    \begin{align*}
        \left(S_{k}-\ka_{1} S_{k}^{11}\right)&  \ka_{1}^{2 m-3} u_{111}^{2}+2 P_{m}\sum_{j\ne 1}  S_{k}^{jj} \sum_{l=1}^{m-4} \ka_{j}^{l} \ka_{1}^{m-2-l} u_{jj1}^{2} \\
        \geq& -(n+2)  S_{k}^{11} \sum_{j \neq 1} \ka_{j}^{2 m-2} u_{j j 1}^{2}+2(m-4)\sum_{j \neq 1}  S_{k}^{j j} \ka_{j}^{2 m-2} u_{j j 1}^{2} \\
        & +C \z[\frac{M^2}{m^2 L^2}+\frac{A^2}{m^2}+\frac{B^2}{m^2 u_{11}^2}\x] \frac{S_{k}-\ka_{1} S_{k}^{11}}{\ka_{1}}  P_{m}^2 u_{11}^{2},\\
        \geq&  C \z[\frac{M^2}{m^2 L^2}+\frac{A^2}{m^2}+\frac{B^2}{m^2 u_{11}^2}\x](S_{k}-\ka_{1} S_{k}^{11})  P_{m}^2 \lam_1,
    \end{align*}
    provided $m$ is sufficiently large.
\end{proof}

\begin{lemma}\label{lem 3.4}
    Suppose that the constants $\lam_1$, $L$, $M$, $m$, $A$, and $B$ satisfy
    \begin{equation}\label{constants condition 1}
        L \gg M \gg B \gg A \gg m \gg 1,\   mB \gg A^2,\  \te{and } \lam_1 \te{ sufficiently large.}
    \end{equation}
    Then, the following inequality holds:
    \begin{equation}
        \frac{B}{2}\sum_i S_k^{ii}+\min\z\{0,C \z[\frac{M^2}{m L^2}+\frac{A^2}{m}+\frac{B^2}{m \lam_1^2}\x](S_{k}-\ka_{1} S_{k}^{11}) \lam_1\x\}\geq 0.
    \end{equation}
\end{lemma}
\begin{proof}
    Without loss of generality, assume $S_k-\ka_1 S_k^{11}<0$. It follows from Lemma \ref{lem 2.5} (b) that $|\lam_k|\leq c_0 K_0$. Moreover, $\sum_i S_k^{ii}=(n-k+1)\si_{k-1}+\al(n-k+2)\si_{k-2}$ and 
    $$
    0>S_k -\ka_1 S_k^{11}=(S_k- \lam_1 S_k^{11})-K_0 S_k^{11}=S_k(\lam|1)-K_0S_{k-1}(\lam|1).
    $$
    
    If $\lam_{k-1}\leq C_0 K_0$ for some $C_0>0$, then $S_k -\ka_1 S_k^{11} \geq -C\lam_2 \cdots \lam_{k-2}K_0^3$ and 
    \begin{align*}
        \frac{B}{2}\sum_i S_k^{ii}&\geq \frac{\al B}{2}(n-k+2)\si_{k-2} \geq \frac{\al B}{2}(n-k+2)\lam_1 \cdots \lam_{k-2}\\ 
        %& \geq C\z[\frac{M^2}{m L^2}+\frac{A^2+B^2}{m}\x] \ka_1 \lam_2\cdots \lam_{k-2}K_0^3 \\
        & \geq C\z[\frac{M^2}{m L^2}+\frac{A^2}{m}+\frac{B^2}{m \lam_1^2}\x]\lam_1|S_k-\ka_1S_k^{11}|.
    \end{align*}
    This follows because $\lam \in \gaa_{k-1}$, and we can choose the constants $\lam_1$, $L$, $M$, $m$, $A$, and $B$  such that they satisfy \eqref{constants condition 1}.

    If $\lam_{k-1}\geq L_0K_0$ for a sufficiently large $L_0>0$, then $|S_k-\ka_1 S_k^{11}|\leq C \lam_2 \cdots \lam_{k-1} K_0^2$. Thus, we have
    \begin{align*}
        \frac{B}{2}\sum_i S_k^{ii}&\geq \frac{ B}{2}(n-k+1)\si_{k-1} \geq \frac{ B}{2}(n-k+1)\z[\lam_1 \cdots \lam_{k-1}-C\lam_1 \cdots \lam_{k-2}K_0\x]\\ 
        & \geq \frac{ B}{4}(n-k+1) \lam_1 \cdots \lam_{k-1} \\
        & \geq C\z[\frac{M^2}{m L^2}+\frac{A^2}{m}+\frac{B^2}{m \lam_1^2}\x]\lam_1 |S_k-\ka_1S_k^{11}|.
    \end{align*}
    Here, the constants $\lam_1$, $L$, $M$, $m$, $A$, and $B$ must also satisfy \eqref{constants condition 1}.
\end{proof}

In the end, we continue to complete the proof of Theorem \ref{thm main thm: elliptic}.
\begin{proof}[\textbf{Proof of Theorem $\ref{thm main thm: elliptic}$}]
    From the previous analysis, equation \eqref{inequality 3} becomes:
    \begin{align}
        0 \geq&   \frac{1}{m} \sum_{i \neq 1} E_{i}+\sum_i \z[M \frac{S_{k}^{ii} u_{i i}}{u}-M \frac{S_{k}^{ii} u_{i}^{2}}{u^{2}}+A S_{k}^{i i} u_{i i}^{2}\x] \label{inequality 4} \\
        & +\frac{B}{2} \sum_{i} S_{k}^{i i}-C_{A,B}-Cm u_{11}.\nonumber
    \end{align}

    Now, we consider the term $- \frac{S_k^{ii} u_{i}^{2}}{u^{2}}$. According to \eqref{eq first derivative of test function} and the Cauchy inequality, for $i \neq 1$,
    \begin{align*}
        -S_{k}^{i i}\left(\frac{u_{i}}{u}\right)^{2}& =-S_{k}^{i i}\left[ \frac{m\sum_j \ka_{j}^{m-1} u_{jji}}{MP_{m}}+\frac{A}{M} u_{i} u_{i i}+\frac{B}{M} x_{i}\right]^{2} \\
        & \geq-(1+\es_0)\frac{m^{2}}{M^{2}P_{m}^{2}} S_k^{i i} \left(\sum_j \ka_{j}^{m-1}  u_{jji}\right)^{2}-C_{\es_0}S_k^{ii} \left[\frac{A^{2}}{M^{2}} u_{ii}^{2}+\frac{B^2}{M^2}\right],
    \end{align*}
    Here $\es_0$ can be arbitrarily small. Hence
    \begin{align}
        \frac{1}{m}\sum_{i\ne 1}E_i-M\sum_{i\ne 1}S_k^{ii}\z(\frac{u_i}{u}\x)^2\geq&  \z(1-(1+\es_0)\frac{m}{M}\x) \frac{m}{P_m^2}\sum_{i\ne 1}S_k^{ii}\z(\sum_j \ka_j^{m-1} u_{jji}\x)^2 \label{inequality 5}\\
        & -C\frac{A^2}{M}\sum_{i \ne 1} S_k^{ii}u_{ii}^2-C\frac{B^2}{M}\sum_{i \ne 1}S_k^{ii}. \nonumber
    \end{align}

    Next, we consider the expression $\sum_i \frac{S_k^{ii}u_{ii}}{u}$. Without loss of generality, we may assume that $\sum_{i} S_{k}^{i i} u_{i i} >0$. Hence
    $$
    0<\sum_i S_{k}^{i i} u_{i i}=\sum_{i}\left(\si_{k}^{i i}+\alpha \sigma_{k-1}^{ii}\right) u_{i i}
    \leq k S_{k} \leq C.
    $$
    This implies that $M\sum_i \frac{S_k^{ii}u_{ii}}{u}\geq \frac{CM}{u}$. By combining \eqref{inequality 4}, \eqref{inequality 5} and the assumption that $(-u)u_{11}\geq L$, we obtain 
    \begin{align}
        0 \geq&  \left(1-(1+\es_0)\frac{ m}{M}\right)\frac{m}{P_{m}^2} \sum_{i\ne 1} S_{k}^{ii}\left(\sum_{j} \ka_{j}^{m-1} u_{jji}\right)^{2}\label{inequality 6} \\
        & + \sum_{i \neq 1} A\left(1-C\frac{ A}{M}\right) S_{k}^{i i} u_{ii}^{2}+\z(A-C\frac{M}{L^2}\x) S_{k}^{11} u_{11}^{2}\nonumber\\
        & +B\z(\frac{1}{2}-C\frac{B}{M}\x)\sum_{i\ne 1}S_k^{ii}+\frac{B}{2}S_k^{11}-C_{A,B}-Cm u_{11} \nonumber\\
        \geq&  \frac{A}{2}\ttt S_k u_{11} -Cm u_{11}-C(A+B) \nonumber \\
        \geq&  \frac{A}{4}\ttt S_k u_{11}-C(A+B),\nonumber
    \end{align}
    where we have used $S_k^{11}u_{11}\geq \ttt S_k$ in the above inequalities, and the constants $L$, $M$, $m$, $A$, $B$ satisfy \eqref{constants condition 1}. Thus there exists $\ga_0$ such that
    $$
    \ga_0 \log (-u(x)) +\log u_{11}(x) \leq   C, \quad \forall \  x \in \ooo,
    $$ where the constant $\ga_0$ and $C$ depend on $n$, $k$, $\inf f$, $\|f\|_{C^{1,1}}$, $\|u\|_{C^1}$, and $\operatorname{diam}(\ooo)$. Hence, we complete the proof of Theorem \ref{thm main thm: elliptic}.
\end{proof}
\begin{remark}
    Due to the lack of \cite[(2.7)]{Liu-Ren-2023-JFA}, we cannot derive the estimate for $\gamma_{0}=1+\epsilon$.
\end{remark}
\begin{remark}
    We can also derive the estimate bu using $\lambda_{1}$ instead of $P_{m}$, however, we use $P_{m}$ since it can produce more convexity.
\end{remark}

\section{Pogorelov estimates for the parabolic equations}\label{sec 4}
\setcounter{equation}{0}

In this section, we investigate the parabolic version of the Pogorelov estimates and prove Theorems \ref{thm main thm: parabolic 1} and \ref{thm main thm: parabolic 2} separately.

\subsection{Proof of Theorem \ref{thm main thm: parabolic 1}}\label{subsec 4.1}

Throughout this subsection, we assume that $u$ is a $(k-1)$-convex-monotone solution of equation \eqref{main parabolic eq}, that is, 
\begin{equation}
    \lam(D^2 u)\in \gaa_{k-1}, \quad  \te{and}\quad 0<m_1 \leq -u_t\leq m_2.
\end{equation}

Firstly, $u\leq 0$ in $\mathcal{O}$ by the maximum principle. We now consider the corresponding test function:
$$
    \phi=\log \lam_{\max}+\bb  \log (-u)+\frac{\es}{2} |Du|^{2},
$$ 
where $\lam_{max}$ is the largest eigenvalue, $\bb=1+2\de>0$, and $\de$ and $\es$ are small parameters.

Assume that $\p$ achieves its maximum at $(x_{0},t_0)$, where $x_0 \in \mathcal{O}(t_0)$. We can choose an orthonormal frame such that $D^2 u(x_0,t_0)$ is diagonalized. Furthermore, we may assume $\lam_{max}$ has multiplicity $l$, i.e.,
$$
\lam_{1}=\cdots =\lam_{l}>\lam_{l+1}\geq \cdots\geq  \lam_{n}, \quad \te{at } (x_0,t_0).
$$
In the remainder of the proof, all computations will be carried out at $(x_{0},t_0)$. According to \cite[Lemma 5]{BCD}, we obtain
\begin{align*}
    & \de_{kj}(\lam_1)_i=u_{kji},\quad 1\leq k,j\leq l, \\
    & (\lam_{1})_{ii}\geq u_{11ii}+2\sum_{j>l}\frac{u_{1ji}^2}{\lam_1-\lam_j},
\end{align*}
in the viscosity sense.
Then, at $(x_0,t_0)$
\begin{equation}\label{eq time derivative of test function: parabolic}
    0\leq \p_t =\frac{u_{11t}}{u_{11}}+\bb \frac{u_t}{u}+\es u_j u_{jt},
\end{equation}
\begin{equation}\label{eq first derivative of test function: parabolic}
    0=\phi_{i}=\frac{u_{11i}}{u_{11}}+\bb  \frac{u_{i}}{u}+\es u_i u_{ii},
\end{equation}
\begin{align}
    & 0 \geq \phi_{i i}\geq \frac{u_{11ii}}{u_{11}}+\sum_{j >l} \frac{2u_{1ji}^2}{u_{11
    }(u_{11}-u_{jj})}-\frac{u_{11i}^2}{u_{11}^2}+\bb \frac{u_{ii}}{u}-\bb \frac{u_i^2}{u^2}+\es u_{ii}^2+\es u_{j} u_{jii}. \label{eq second derivative of test function: parabolic} 
\end{align}
Contracting \eqref{eq second derivative of test function: parabolic} with $S_k^{ii}$, we have
\begin{align}
    0 \geq S_k^{ii}\p_{ii} \geq&  \frac{S_k^{ii}u_{11ii}}{u_{11}} +\sum_{j>l}\frac{2S_k^{11}u_{11j}^2}{u_{11}(u_{11}-u_{jj})}+ \sum_{j>l} \frac{2S_k^{jj} u_{jj1}^2}{u_{11}(u_{11}-u_{jj})} -\frac{S_k^{ii}u_{11i}^2}{u_{11}^2} \label{eq 4.4} \\
    &+\bb \frac{S_k^{ii}u_{ii}}{u}-\bb \frac{S_k^{ii}u_i^2}{u^2} +\es S_{k}^{i i} u_{ii}^2+\es S_k^{ii}u_ju_{jii}.\nonumber
\end{align}

Differentiating equation \eqref{main parabolic eq} gives
\begin{align}
    -u_{tj} S_k-u_t  S_{k}^{ii} u_{ii j} &=f_{x_j}+f_{u} u_{j},  \label{First derivation of the main eq: parabolic} \\
    -u_{t11}S_k-2 u_{t1} S_{k}^{ii} u_{ii 1} &-u_t S_{k}^{pq, r s} u_{p q 1} u_{r s 1} -u_t S_k^{ii}u_{ii11} \label{second derivation of the main eq: parabolic}\\
    &=f_{x_1 x_1}+2f_{x_1 u}u_1+f_{u u} u_{1}^{2}+f_{u}u_{11}. \nonumber
\end{align}
Hence 
\begin{align*}
    \frac{S_k^{ii}u_{ii11}}{u_{11}} & \geq  \underbrace{ -\frac{u_{t11}}{u_t u_{11} }S_k }_{I_1} \underbrace{  -2\frac{u_{t1}}{u_t u_{11} }S_k^{ii}u_{ii1}}_{I_2}  \underbrace{-\frac{S_k^{pp,qq}u_{pp1}u_{qq1}}{u_{11}}  }_{I_3}+\underbrace{ \frac{S_k^{pp,qq}u_{pq1}^2}{u_{11}} }_{I_4}-C.
\end{align*}

By \eqref{eq time derivative of test function: parabolic}, \eqref{eq first derivative of test function: parabolic}, \eqref{First derivation of the main eq: parabolic} and the Cauchy inequality, we have 
$$
I_1\geq S_k\z[\frac{\bb}{u}+\es\frac{u_j u_{jt}}{u_t}\x], \quad I_2\geq  \frac{1}{u_{11}}\z[(2-\es_0)\frac{(\sum_j S_k^{jj}u_{jj1})^2}{S_k}-C_{\es_0}\x],
$$
and
$$
I_1+\es S_k^{ii}u_j u_{jii} \geq S_k\frac{\bb}{u}+\es u_j \z[\frac{S_ku_{jt}}{u_t}+S_k^{ii}u_{jii}\x] \geq   \frac{C}{u}-C.
$$

Applying Lemma \ref{lem 2.3} (e), we have for small constant $\es_0>0$,
$$
I_2+I_3\geq \frac{1}{u_{11}}\z[(2-\es_0)\frac{(\sum_j S_k^{jj}u_{jj1})^2}{S_k}-S_{k}^{pp,qq}u_{pp1}u_{qq1}-C_{\es_0}\x]\geq -C.
$$
Moreover, $I_4\geq 2\sum_{j\ne 1} S_k^{jj,11}\frac{u_{11j}^2}{u_{11}}$, and $\bb \frac{S_k^{ii}u_{ii}}{u}=\frac{\bb}{u}[S_k-\al \si_{k-1}]\geq \frac{C}{u}$. Therefore, \eqref{eq 4.4} becomes
\begin{align}
    0\geq&   \frac{C}{u}-C +2 \sum_{j\ne 1} S_k^{jj,11}\frac{u_{11j}^2}{u_{11}}+2 \sum_{j>l} \frac{S_k^{11}u_{11j}^2}{u_{11}(u_{11}-u_{jj})} \label{eq 4.7}\\
    & -\frac{S_k^{ii}u_{11i}^2}{u_{11}^2}-\bb \frac{S_k^{ii}u_i^2}{u^2}+\es S_k^{ii} u_{ii}^2.\nonumber 
\end{align}

Next, we divide into two cases to deal with \eqref{eq 4.7}.

\textbf{Case 1:} $\lam_n <-\frac{\de \lam_1}{3}$ for some small $\de>0$. Using \eqref{eq first derivative of test function: parabolic} and the Cauchy inequality, we have 
\begin{equation}\label{eq 4.8}
    -\frac{S_k^{ii}u_{11i}^2}{u_{11}^2} \geq -2\bb^2 S_k^{ii}\z(\frac{u_{i}}{u}\x)^2-2\es^2 S_k^{ii}(u_i u_{ii})^2,\quad \te{for any } 1\leq i\leq n.
\end{equation}
Hence, \eqref{eq 4.7} becomes
\begin{align*}
    0& \geq \frac{C}{u}-C   
    -\bb(1+2\bb) S_k^{ii}\z(\frac{u_{i}}{u}\x)^2+\es \sum_i (1-2\es u_i^2 )S_k^{ii} u_{ii}^2\\
    & \geq \frac{C}{u}-C-\bb(1+2\bb) S_k^{ii}\z(\frac{u_{i}}{u}\x)^2+\frac{\es}{2} \sum_i  S_k^{ii} u_{ii}^2\\
    & \geq \frac{C}{u}-C+\z[\frac{\es \de^2}{18}u_{11}^2-C\bb(1+2\bb)\x]S_k^{nn} +\frac{\es}{2}  S_k^{11} u_{11}^2\\
    & \geq \frac{C}{u}-C+\frac{\es}{2}\ttt S_k u_{11}.
\end{align*}
Here, $\ttt$ is in Lemma \ref{lem 2.3} (b) and we choose $\es$ sufficiently small, such that $1 \geq 4\es\|Du\|_{L^\wq}^2$. Then, we get \eqref{eq main result: parabolic}.

\textbf{Case 2:} $\lam_n \geq -\frac{\de \lam_1}{3}$. Applying \eqref{eq first derivative of test function: parabolic}, we get, for any $j>l$
$$
\z(\frac{u_j}{u}\x)^2\leq \frac{1+\de}{\bb^2}\frac{u_{11j}^2}{u_{11}^2}+\frac{1+\de^{-1}}{\bb^2}(\es u_j u_{jj})^2,
$$
hence
\begin{align*}
    -\frac{S_k^{jj}u_{11j}^2}{u_{11}^2}-\bb\frac{S_k^{jj}u_j^2}{u^2}& \geq -\z(1+\frac{1+\de}{\bb}\x)\frac{S_k^{jj}u_{11j}^2}{u_{11}^2}-\frac{1+\de^{-1}}{\bb}S_k^{jj}(\es u_j u_{jj} )^2.
\end{align*}
Moreover, for any $j>l$,
\begin{align*}
    \frac{2S_k^{11,jj}u_{11j}^2}{u_{11}}+\frac{2S_k^{11}u_{11j}^2 }{u_{11}(u_{11}-u_{jj})}=\frac{2S_k^{jj}u_{11j}^2}{u_{11}(u_{11}-u_{jj})}& \geq \frac{6S_k^{jj}u_{11j}^2}{(3+\de)u_{11}^2}>\z(1+\frac{1+\de}{\bb}\x)\frac{S_k^{jj}u_{11j}^2}{u_{11}^2}.
\end{align*}
On the other hand, for any $i\leq l$, we use \eqref{eq 4.8}, and hence \eqref{eq 4.7} becomes 
\begin{align*}
    0 \geq&  \frac{C}{u}-C -\frac{1+\de^{-1}}{\bb}\sum_{j>l}S_k^{jj}(\es u_j u_{jj})^2-2\sum_{i\leq l} \z[\bb^2 S_k^{ii}\z(\frac{u_i}{u}\x)^2+\es^2S_k^{ii}(u_i u_{ii})^2\x] \\
    & -\bb \sum_{i\leq l} S_k^{ii}\z(\frac{u_i}{u}\x)^2 +\es \sum_i S_k^{ii}u_{ii}^2 \\
    \geq&  \frac{C}{u}-C+\es\z(1-\es \bb^{-1}\z(1+\de^{-1}\x) \|Du\|_{L^\wq}^2\x)\sum_{j>l}S_k^{jj} u_{j j}^2 \\
    & -\bb (1+2\bb)\sum_{i\leq l}S_k^{ii}\z(\frac{u_i}{u}\x)^2+\es(1-2\es \|Du\|_{L^\wq}^2)\sum_{i\leq l}S_k^{ii}u_{ii}^2.
\end{align*}
Now, we choose $\es$ sufficiently small, such that 
$$
\es >\max\{4,2(1+\de^{-1})\bb^{-1}\}\es^2\|Du\|_{L^\wq}^2.
$$
Finally, applying Lemma \ref{lem 2.3} (b) yields
\begin{align*}
    \frac{\es}{4}\ttt S_k u_{11}+\frac{\es}{4} u_{11}^2\sum_{i\leq l} S_k^{ii}\leq \frac{\es}{2} u_{11}^2\sum_{i\leq l} S_k^{ii}  \leq \frac{C}{-u}+C+\frac{C}{u^2}\sum_{i\leq l}S_k^{ii}.
\end{align*}
Hence, we complete the proof of Theorem \ref{thm main thm: parabolic 1}.

\subsection{Proof of Theorem \ref{thm main thm: parabolic 2}}\label{subsec 4.2}

Assuming that $u$ is a semi-convex solution of equation \eqref{main parabolic eq} with 
$$
\lam(D^2 u) \in \gaa_{k-1}, \quad \te{and} \quad 0<m_1\leq -u_t\leq m_2.
$$
By the maximum principle, we have $u\leq 0$. 

We now consider the following test function:
$$
    \phi=M  \log (-u)+\log P_m+\frac{A}{2} |Du|^{2}+\frac{B}{2}|x|^2,
$$ 
where $P_{m}=\sum_{i=1}^{n} \kappa_{i}^{m}$, $\kappa_{i}=\lambda_{i}+K_{0}>0$, and $M, m, A, B>0$ are parameters to be determined later.

Assume that $\p$ achieves its maximum at $(x_{0},t_0)$ with $x_0 \in \mathcal{O}(t_0)$. Without loss of generality, we assume $D^2 u(x_0,t_0)$ is diagonalized and 
$$
u_{11}\geq u_{22} \geq \cdots \geq u_{nn}, \quad \te{at} \ (x_0,t_0).
$$
In the remainder of the proof, all computations will be carried out at $(x_{0},t_0)$. Similar to Subsection \ref{subsec 4.1}, the following inequality holds:
\begin{align}
    0\leq \p_t  =&M \frac{u_t}{u}+\frac{m}{P_m}\ka_j^{m-1}u_{jjt}+A u_j u_{jt}, \label{eq time derivative of test function: parabolic 2} \\
    0=\phi_{i}  =&M\frac{u_{i}}{u}+\frac{m}{P_m}\ka_j^{m-1}u_{jji}+A u_i u_{ii}+Bx_i, \label{eq first derivative of test function: parabolic 2}\\
    0 \geq \phi_{i i}  =&M \frac{u_{ii}}{u}-M \frac{u_i^2}{u^2}+\frac{m(m-1)}{P_m}\ka_{j}^{m-2}u_{jji}^2 -\frac{m^2}{P_m^2}\z(\sum_{j}\ka_j^{m-1}u_{jji}\x)^2 \label{eq second derivative of test function: parabolic 2} \\
    & +\frac{m}{P_m}\ka_j^{m-1}\z[u_{jjii}+2\sum_{p\ne j}\frac{u_{jpi}^2}{\ka_j -\ka_p}\x] +Au_{ii}^2 +Au_{j}u_{jii}+B, \nonumber \\
    -u_{tj} S_k - &u_t  S_{k}^{ii} u_{ii j} =f_{x_j}+f_{u} u_{j}+f_{p_j}u_{jj},  \label{First derivation of the main eq: parabolic 2} \\
    -u_{tjj}S_k  -&2 u_{tj} S_{k}^{ii} u_{iij} -u_t S_{k}^{pq, r s} u_{p q j} u_{r s j} -u_t S_k^{ii}u_{iijj} \geq f_{p_s}u_{sjj}-C-Cu_{11}^2. \label{second derivation of the main eq: parabolic 2}
\end{align}
Then, contracting $S_k^{ii}$ in \eqref{eq second derivative of test function: parabolic 2}, we have
\begin{align}
    0 \geq&  A S_k^{ii}u_{ii}^2+\underbrace{A S_k^{ii}u_{j}u_{jii}}_{\operatorname{I}} +B\sum_i S_k^{ii} +MS_k^{ii}\frac{u_{ii}}{u}-MS_k^{ii}\z(\frac{u_i}{u}\x)^2  \label{0 geq ... parabolic 2} \\
    & +\frac{m(m-1)}{P_m}S_k^{ii}\ka_{j}^{m-2}u_{jji}^2-\frac{m^2}{P_m^2}S_k^{ii} \z(\sum_{j}\ka_j^{m-1}u_{jji}\x)^2  \nonumber \\ 
    & +\underbrace{\frac{m}{P_m}S_k^{ii}\ka_j^{m-1}u_{jjii}}_{\operatorname{II}} +\frac{m}{P_m}S_k^{ii}\sum_{p\ne j}\frac{\ka_p^{m-1}-\ka_j^{m-1}}{\ka_p-\ka_j} u_{pji}^2. \nonumber
\end{align}
Using \eqref{second derivation of the main eq: parabolic 2}, we have
\begin{align*}
    \operatorname{II} %&=\frac{m}{P_m} S_k^{ii}\sum_{j}\ka_j^{m-1}u_{jjii} \\ 
    \geq&  \underbrace{ -\frac{m}{P_m}\sum_j \ka_j^{m-1}S_k^{pq,rs}u_{pqj}u_{rsj} }_{\operatorname{II}_1} +\underbrace{\frac{m}{u_t P_m}\sum_j \ka_j^{m-1}\z[C+Cu_{11}^2-f_{p_s}u_{sjj}\x]  }_{\operatorname{II}_2} \\
    & \underbrace{ -\frac{m S_k}{u_t P_m} \sum_j u_{tjj}\ka_j^{m-1} }_{\operatorname{II}_3} \underbrace{ -\frac{2m}{P_m}\sum_j \ka_j^{m-1} S_k^{ii} u_{iij} \frac{u_{tj}}{u_t}}_{\operatorname{II}_4}.
\end{align*} 
We now compute $\operatorname{I}$ and $\operatorname{II}$.
\begin{align*}
    \operatorname{II}_1 &=-\frac{m}{P_m}\sum_j \ka_j^{m-1} \z[S_k^{pp,qq}u_{ppj}u_{qqj}-S_k^{pp,qq}u_{pqj}^2\x] \\
    & \geq \frac{m}{P_m}\sum_j \ka_j^{m-1} \z[-S_k^{pp,qq}u_{ppj}u_{qqj}+2 \sum_{p\ne j}S_k^{pp,jj}u_{pjj}^2\x].
\end{align*}
By applying \eqref{eq time derivative of test function: parabolic 2}, \eqref{eq first derivative of test function: parabolic 2}, \eqref{First derivation of the main eq: parabolic 2} in sequence, along with the Cauchy inequality, we obtain
\begin{align*}
    \operatorname{II}_3 & \geq M\frac{S_k}{u}+AS_k\frac{u_j u_{jt}}{u_t}, \quad  \operatorname{II}_2 \geq -CB+C\frac{M}{u}-Cmu_{11}+A\frac{f_{p_i}u_i u_{ii}}{u_t}, \\
    \operatorname{II}_2 & + \operatorname{II}_3 +\operatorname{I} \geq C\frac{M}{u}-C(A+B)-Cm u_{11}, \\
    \operatorname{II}_4 & =\frac{2m}{P_m}\sum_j \ka_j^{m-1} \z(\sum_i S_k^{ii}u_{iij}\x) \z[\frac{\sum_i S_k^{ii}u_{iij}}{S_k}+\frac{(f)_j}{u_t S_k}\x] \\
    & \geq \frac{m}{P_m}\sum_j \ka_j^{m-1}\z[ (2-\es)  \frac{\z(\sum_i S_k^{ii}u_{iij}\x)^2}{S_k}-C_\es  \frac{(f)_j^2}{u_t^2 S_k}\x],
\end{align*}
where $\es>0$ is a sufficiently small constant. Hence  
\begin{align*}
    \operatorname{II}_1 +\operatorname{II}_4  \geq \frac{m}{P_m}\sum_j \ka_j^{m-1} \z[ (2-\es)\frac{(S_k)_j^2}{S_k}-S_k^{pp,qq}u_{ppj}u_{qqj} +2\sum_{i\ne j}S_k^{ii,jj}u_{ijj}^2\x] -Cm u_{11}.
\end{align*}
Therefore, \eqref{0 geq ... parabolic 2} becomes
\begin{align*}
    0 \geq & AS_k^{ii}u_{ii}^2+B\sum_i S_k^{ii}+ M\z[\frac{S_k^{ii}u_{ii}}{u}+\frac{C}{u}\x]-MS_k^{ii}\z(\frac{u_i}{u}\x)^2-C_{A,B}-Cmu_{11}\\
    %& -C(m+A+B)-C(m+A)u_{11} \\
    & +\frac{m(m-1)}{P_m}\sum_iS_k^{ii}\sum_j \ka_j^{m-2} u_{jji}^2 -\frac{m^2}{P_m^2}\sum_i S_k^{ii} \z(\sum_{j}\ka_j^{m-1}u_{jji}\x)^2\\
    &+ \frac{m}{P_m}\sum_i \ka_i^{m-1}\z[(2-\es)\frac{(S_k)_i^2}{S_k}-S_k^{pp,qq}u_{ppi}u_{qqi}\x] +\frac{2m}{P_m}\sum_i \sum_{j\ne i}S_k^{ii,jj} \ka_j^{m-1} u_{jji}^2 \\
    &   +\frac{2m}{P_m}\sum_i S_k^{jj}\sum_{j \ne i}\frac{\ka_j^{m-1}-\ka_i^{m-1}}{\ka_j-\ka_i} u_{jji}^2, 
\end{align*}
where $C_{A,B}\leq C (A+B)$.

Based on the analysis in Section \ref{sec 3}, we obtain an inequality similar to that in inequality \eqref{inequality 5}. It is important to note that, when $\es$ is sufficiently small, the term $\frac{m}{P_m} \ka_i^{m-1}\z[(2-\es)\frac{(S_k)_i^2}{S_k}-S_k^{pp,qq}u_{ppi}u_{qqi}\x]$ is treated in the same way as $A_i$. Hence,  by choosing the constants that satisfy \eqref{constants condition 1}, there exists $\ga_0$ such that 
$$
    \ga_0 \log (-u(x)) +\log u_{11}(x) \leq C, \quad \forall \  x \in \mathcal{O},
$$ 
where $\ga_0$ and $C$ depend on $n$, $k$, $m_1$, $m_2$, $\inf f$, $\|f\|_{C^{1,1}}$, $\|u\|_{C^1}$, and $\operatorname{diam}(\mathcal{O}(0))$. Hence, we complete the proof of Theorem \ref{thm main thm: parabolic 2}.

\section{The rigidity theorems for the elliptic Hessian equation and its parabolic version}\label{sec 5}
\setcounter{equation}{0}

In this section, we prove the rigidity theorem. In order to get rid of the dependence of $Du$ in Pogorelov estimates, we need the following lemma.
\begin{lemma}\label{lem 5.1}
    Let $u$ be a $(k-1)$-convex and semi-convex solution to
    \begin{equation}\label{eq 5.1}
        \left\{\begin{aligned}
            & \si_k(D^2 u)+\al \si_{k-1}(D^2 u)=f(x), && \te{in}\ \ooo, \\
            & u=0, && \te{on} \ \pa \ooo,
        \end{aligned}\right.
    \end{equation}
    where $f \in C^{1,1}(\overline{\ooo})$ is a positive function and $\al >0$. Then there exists $\ga_0>0$ such that
    \begin{equation}\label{eq 5.2}
        (-u)^{\ga_0} \dd u \leq C,
    \end{equation}
    where $\ga_0$ and  $C$ depend on $n$, $k$, $\inf f$, $\|f\|_{C^{1,1}}$, and $\operatorname{diam}(\ooo)$.
\end{lemma}
\begin{proof}
    Obviously, for sufficiently large $a_0$ and $b_0$, the function $w=\frac{a_0}{2}|x|^2-b_0$ can control $u$ by comparison principal (see \cite{CNS3} for detail), namely, 
    $$
    w\leq u\leq 0.
    $$
    Here, $a_0$, $b_0$ depend on $\operatorname{diam}(\ooo)$. Hence, in the following proof, the constants $\ga_0$ and $C$ in \eqref{eq 5.2} can contain $\sup_{\ooo} |u|$.

    Next, we consider the following test function
    $$\phi=M \log (-u)+\log P_m+\frac{B}{2}|x|^2,$$ 
    where $P_m=\sum_{j} \ka_j^m$, $\ka_j=\lambda_j+K_0>0$, and $M, m, B>0$ are parameters to be determined later. Assume $\p$ attains its maximum at $x_0 \in \ooo$. Then, at $x_0$
    \begin{align}
        0=\phi_i & =M \frac{u_i}{u}+ \frac{m}{P_m} \ka_j^{m-1} u_{j j i}+B x_i, \label{eq 5.3}\\
        0 \geq \phi_{i i} =&M \frac{u_{i i}}{u}-M\left(\frac{u_i}{u}\right)^2+m \frac{\ka_j^{m-1}}{P_m}\left[u_{j j i i}+2 \sum_{p \neq j} \frac{u_{j p i}^2}{\ka_j-\ka_p}\right] \label{eq 5.4}\\
        & + \frac{m(m-1)}{P_m}\ka_j^{m-2} u_{j j i}^2- \frac{m^2}{P_m^2}\left(\sum_j \ka_j^{m-1} u_{jji}\right)^2+B.\nonumber
    \end{align}
    Multiplying $S_k^{ii}$, then
    \begin{align}
        0 \geq&   M S_k^{ii}\frac{u_{ii}}{u}-M S_k^{ii}\z(\frac{u_i}{u}\x)^2 +\frac{m}{P_m}S_k^{ii}\ka_j^{m-1}u_{jjii}+\frac{2m}{P_m}S_k^{ii}\sum_{j\ne i}\frac{\ka_j^{m-1}-\ka_i^{m-1}}{\ka_j -\ka_i}u_{jii}^2 \label{eq 5.5}\\
        & +\frac{m(m-1)}{P_m} S_k^{i i} \ka_j^{m-2} u_{j ji}^2-\frac{m^2}{P_m^2} S_k^{i i}\left(\sum_j \ka_j^{m-1} u_{jj i}\right)^2+B\sum_i S_k^{ii}.\nonumber
    \end{align}
    Differential \eqref{eq 5.1} twice, we have
    \begin{align}
        & \left(S_k\right)_j=S_k^{i i} u_{i ij}=f_j, 
        \quad S_k^{pq,rs} u_{pqj} u_{rsj}+S_k^{ii} u_{iijj}=f_{j j}. \label{eq 5.6}
    \end{align}
    By \eqref{eq 5.3}, \eqref{eq 5.6}, and the Cauchy inequality, we obtain, for any $1\leq i\leq n$
    \begin{align}
        & -MS_k^{ii}\z(\frac{u_i}{u}\x)^2 %& = -\frac{1}{M}S_k^{ii}\z(\frac{m}{P_m}\ka_j^{m-1}u_{jji}+B x_i\x)^2 
        \geq -\frac{2}{M}\frac{m^2}{P_m^2} S_k^{ii} \z(\sum_j \ka_j^{m-1}u_{jji}\x)^2-\frac{2B^2}{M}S_k^{ii} x_i^2, \label{eq 5.7}
    \end{align}
    and
    \begin{align*}
        S_k^{i i} u_{ii j j} %=f_{j j}-S_k^{pq,rs} u_{p q j} u_{r s j} 
        \geq-C-S_k^{p p, q q} u_{p p j} u_{q q j}+S_k^{p p, q q} u_{p q j}^2 \geq -C-S_k^{p p, q q} u_{p p j} u_{q q j}+2 \sum_{p \neq j} S_k^{p p,jj} u_{p j j}^2. 
    \end{align*}
    Hence
    \begin{align}\label{eq 5.8}
        & \frac{m}{P_m} S_k^{i i} \ka_j^{m-1}u_{jjii}\geq -C\frac{m}{u_{11}}-\frac{m}{P_m}\ka_j^{m-1}S_k^{pp,qq}u_{ppj}u_{qqj}+\frac{2m}{P_m}\sum_{j\ne i}S_k^{ii,jj}\ka_j^{m-1}u_{jji}^2. 
    \end{align}

    Using the fact that $\frac{(S_k)_j^2}{S_k}\leq C$ and $\sum_i S_k^{ii}\frac{u_{ii}}{u}\geq k\frac{S_k}{u}$, and substituting \eqref{eq 5.7} and \eqref{eq 5.8} into \eqref{eq 5.5}, we conclude that
    \begin{align}
        0 \geq&  kM\frac{S_k}{u} +B\z(1-C\frac{B}{M}\x) \sum_i S_k^{i i}-C \frac{m}{u_{11}} \nonumber \\
        & +\frac{m}{P_m} \sum_i \ka_i^{m-1} \left[2\frac{\left(S_k\right)_i^2}{S_k} - S_k^{p p,qq} u_{p p i} u_{qq i}\right]+\frac{2m}{P_m}\sum_i \sum_{j \neq i} S_k^{i i,j j}\ka_j^{m-1} u_{jji}^2 \nonumber \\
        & +\frac{m(m-1)}{P_m}\sum_i \sum_j S_k^{ii}\ka_j^{m-2}u_{jji}^2 +\frac{2m}{P_m}\sum_i \sum_{j \ne i}S_k^{jj} \frac{\ka_j^{m-1}-\ka_i^{m-1}}{\ka_j -\ka_i }u_{jji}^2 \nonumber\\
        & - \z(1+\frac{2}{M}\x)\frac{m^2}{P_m^2}\sum_i S_k^{ii}\z(\sum_j \ka_j^{m-1}u_{jji}\x)^2. \nonumber\\
        \triangleq&  kM\frac{S_k}{u} +B\z(1-C\frac{B}{M}\x) \sum_i S_k^{i i}-C \frac{m}{u_{11}} \label{eq 5.9} \\
        & +\sum_i\z[A_i+B_i+C_i+D_i-\z(1+\frac{2}{M}\x)E_i\x],\nonumber
    \end{align}
    where $A_i$, $B_i$, $C_i$, $D_i$, $E_i$ are defined in Section \ref{sec 3}. 

    Similar to the analysis in Section \ref{sec 3}, for a fixed sufficiently small constant $\e$ in Lemma \ref{lem A_1+... geq ...}, we choose constants $\lam_1$, $L$, $M$, $m$, $B$ such that
    \begin{equation}\label{constants condition 3}
        L\gg  M \gg B\gg  m\gg 1, \ \te{and} \  \lam_1 \  \te{is sufficiently large}.
    \end{equation}
    And combining with Lemma \ref{lem 2.5} (3), we have
    \begin{align}
        -C\frac{M}{u} & \geq -C\frac{m}{u_{11}} +\frac{B}{2}\sum_i S_k^{ii}+\sum_{i\ne 1}\z(\frac{1}{m}-\frac{2}{M}\x)E_i +\z(\frac{\e}{m}-\frac{2}{M}\x)E_1 \nonumber\\
        & \geq -C\frac{m}{u_{11}}+\frac{B}{2}\sum_i S_k^{ii} %= C\frac{m}{u_{11}}+\frac{B}{4} \z[(n-k+1)\si_{k-1}+\al(n-k+2)\si_{k-2}\x].
        \geq - C\frac{m}{u_{11}}+c_0 \frac{B}{2}\si_1^{\frac{1}{k-2}}.\nonumber \\
        & \geq - C\frac{m}{u_{11}}+c_0 \frac{B}{2}u_{11}^{\frac{1}{k-2}}. \label{eq 5.11}
    \end{align}
    Here, $c_0$ is in Lemma \ref{lem 2.5}. Therefore, we complete the proof.
\end{proof}

\begin{proof}[\textbf{Proof of Theorem $\ref{thm rigidity theorem 1}$}]
    The proof is standard (refer to \cite{He-Sheng-Xiang-Zhang-2022-CCM,LRW2,Liu-Ren-2023-JFA,TW,Tu-2024-arxiv,Zhang-2024-arxiv} for details). Suppose $u$ is an entire solution of the equation \eqref{main eq 1}. For arbitrary positive constant $R>1$, define
    $$
    v(y):=\frac{u(Ry)-R^2}{R^2}, \quad  \ooo_R:=\{y\in \rr^n: u(Ry)\leq R^2\}.
    $$
    Then, $v$ satisfies the following Dirichlet problem:
    \begin{equation*}
        \left\{\begin{aligned}
            & S_k(D^2 v)=1, && \te{in} \ \ooo_R, \\
            & v=0, && \te{on} \ \pa \ooo_R.
        \end{aligned}\right.
    \end{equation*}
    By Lemma \ref{lem 5.1}, we have the following estimates
    \begin{equation}\label{eq 5.12}
        (-v)^{\ga_0}\dd v\leq C,
    \end{equation}
    where the constants $\ga_0$ and $C$ depend on $n$, $k$ and the $\operatorname{diam}(\ooo_R)$. By employing the quadratic growth condition \eqref{growth condition}, it follows that
    $$
    a|Ry|^2-b\leq u(Ry) \leq R^2, \quad \te{in} \ \ooo_R.
    $$
    Namely
    $$
    |y|^2 \leq \frac{1+b}{a}.
    $$
    Hence, $\ooo_R$ is bounded. Thus, the constant $\ga_0$ and $C$ are independent of $\operatorname{diam}(\ooo_R)$. Next, we consider the domain
    $$\ooo_R'=\{y: u(Ry)\leq R^2/2\} \subset\ooo_R.$$
    In $\ooo_R'$, we have
    $$
    v(y)\leq-\frac{1}{2}.
    $$
    Hence, it follows form \eqref{eq 5.12} that 
    \begin{equation*}
        \dd v\leq 2^{\ga_0} C, \quad \te{in} \ \ooo_R'.
    \end{equation*}
    Hence $\dd u(x)=\dd v(y)\leq 2^{\ga_0}C$. That means
    $$
    \dd u(x)\leq C_0, \quad \te{in} \ \ooo_R''=\{x:u(x)\leq R^2/2\},
    $$
    where $C_0$ is an absolutely constants. The estimate above holds true for all large $R$. Thus we using Evans-Krylov theory \cite{GT}, we have for any $0<\bb<1$,
    $$
    \|D^2 u\|_{C^{0,\bb}(\ooo_R''\cap B_{R/2})}\leq C(n,\bb)\frac{\|D^2 u\|_{L^\wq(\ooo_R''\cap B_R)}}{R^\bb}\leq\frac{C(\bb)}{R^\bb}.
    $$
    Letting $R\to \wq$, and we finish the proof of Theorem \ref{thm rigidity theorem 1}.
\end{proof}

Finally, we also have the parabolic version of Lemma \ref{lem 5.1}.

\begin{lemma}\label{lem 5.2}
    Let $u$ be a $(k-1)$-convex and semi-convex solution to
    \begin{equation}
        \left\{\begin{aligned}
            & -u_t\cdot S_k\left(D^2 u\right)=f(x),  &&\te{in } \mathcal{O}, \\
            & u=0, && \text{on } \pa \mathcal{O}.
    \end{aligned}\right.
    \end{equation}
    where $f \in C^{1,1}(\overline{\mathcal{O}})$ is a positive function. Assume $0<m_1\leq -u_t\leq m_2$. 
    Then there exists $\ga_0>0$ such that
    \begin{equation}
        (-u)^{\ga_0} \dd u \leq C,
    \end{equation}
    where $\ga_0$ and $C$  depend on $n$, $k$, $m_1$, $m_2$, $\inf f$, $\|f\|_{C^{1,1}}$, $\operatorname{diam}(\mathcal{O}(0))$.
\end{lemma}

\begin{proof}
    The same as Lemma \ref{lem 5.1}, we get the $C^0$ bound of $u$. Now consider the following test function
    $$\phi=M \log (-u)+\log P_m+\frac{B}{2}|x|^2,$$
    where $P_m=\sum_j \ka_j^m, \ka_j=\lambda_j+K_0>0$, and $M, m, B>0$ are parameters to be determined later. Assume $\p$ attains its maximum at $\left(x_0, t_0\right), x_0 \in \mathcal{O}\left(t_0\right)$. 
    Similar to Subsection \ref{subsec 4.2},  at $(x_0,t_0)$, the following inequality holds:
    \begin{align}
        0\leq \p_t =& M \frac{u_t}{u}+\frac{m}{P_m}\ka_j^{m-1}u_{jjt}, \label{eq 5.15} \\
        0=\phi_{i} =& M\frac{u_{i}}{u}+\frac{m}{P_m}\ka_j^{m-1}u_{jji}+Bx_i, \label{eq 5.16}\\
        0 \geq \phi_{i i} = &M \frac{u_{ii}}{u}-M \frac{u_i^2}{u^2}+\frac{m(m-1)}{P_m}\ka_{j}^{m-2}u_{jji}^2 -\frac{m^2}{P_m^2}\z(\sum_{j}\ka_j^{m-1}u_{jji}\x)^2 \label{eq 5.17} \\
        & +\frac{m}{P_m}\ka_j^{m-1}\z[u_{jjii}+2\sum_{p\ne j}\frac{u_{jpi}^2}{\ka_j -\ka_p}\x] +B, \nonumber \\
        -u_{tj} S_k  -&u_t  S_{k}^{ii} u_{ii j} =f_{j},  \label{eq 5.18} \\
        -u_{tjj}S_k  -&2 u_{tj} S_{k}^{ii} u_{iij} -u_t S_{k}^{pq, r s} u_{p q j} u_{r s j} -u_t S_k^{ii}u_{iijj} = f_{jj}. \label{eq 5.19}
    \end{align}
    Then
    \begin{align}
        0\geq &  M S_k^{i i} \frac{u_{i i}}{u}-M S_k^{i i}\left(\frac{u_i}{u}\right)^2+\frac{m}{P_m} S_k^{i i} \ka_j^{m-1} u_{j j i i}+\frac{2m}{P_m} S_k^{i i} \sum_{j \neq i} \frac{\ka_j^{m-1}-\ka_i^{m-1}}{\ka_j-\ka_i} u_{jii}^2 \label{eq 5.20} \\
        & +\frac{m(m-1)}{P_m} S_k^{i i} \ka_j^{m-2} u_{j j i}^2-\frac{m^2}{P_m^2} S_k^{i i}\left(\sum_j \ka_j^{m-1} u_{jji}\right)^2+B \sum_i S_k^{i i}.\nonumber
    \end{align}
    
    It is easy to see that $\sum_i S_k^{ii}\frac{u_{ii}}{u}\geq k\frac{S_k}{u}$. By \eqref{eq 5.16} and the Cauchy inequality, 
    \begin{equation}\label{eq 5.21}
        -MS_k^{ii}\z(\frac{u_i}{u}\x)^2 %=-\frac{1}{M}S_k^{ii}\z[\frac{m}{P_m}\ka_j^{m-1}u_{jji}+Bx_i\x]^2
        \geq -\frac{2}{M}\frac{m^2}{P_m^2}S_k^{ii}\z(\sum_j \ka_j^{m-1}u_{jji}\x)^2-\frac{2B^2}{M}S_k^{ii}x_i^2.
    \end{equation}
  
    Additionally, using \eqref{eq 5.19},
    \begin{align*}
        \frac{m}{P_m}\ka_j^{m-1} S_k^{ii} u_{jjii} & =\frac{m}{P_m} \ka_j^{m-1} \left[-\frac{u_{t j j}}{u_t} S_k-2 \frac{u_{t j}}{u_t} S_k^{i i} u_{iij}- S_k^{pq,rs} u_{p q j} u_{r s j}-\frac{f_{jj}}{u_t}\right]\\
        & \triangleq Q_1+Q_2+Q_3+Q_4.
    \end{align*}
    Using \eqref{eq 5.15}, \eqref{eq 5.18} and the Cauchy inequality, respectively
    \begin{align*}
        & Q_1 \geq M\frac{S_k}{u}, \\
        & Q_2 \geq \frac{m}{P_m}\ka_j^{m-1}(2-\es) \frac{\z(\sum_i S_k^{ii}u_{iij}\x)^2}{S_k} -C_\es \frac{m}{u_{11}}. \\
        & Q_3 \geq \frac{m}{P_m}\ka_j^{m-1} \z[-S_k^{pp,qq}u_{ppj}u_{qqj}+2\sum_{j \ne p} S_k^{pp,jj}u_{pjj}^2 \x].
    \end{align*}
    Substituting these computations into \eqref{eq 5.20}, we obtain an inequality same to that in \eqref{eq 5.9}:
    \begin{align*}
        -CM\frac{S_k}{u} \geq   \frac{B}{2} \sum_i S_k^{i i}-C \frac{m}{u_{11}}  +\sum_i\z[A_i+B_i+C_i+D_i-\z(1+\frac{2}{M}\x)E_i\x],
    \end{align*}
    In view of \eqref{eq 5.21}, the constant $C$ here depends on $\operatorname{diam}(\mathcal{O}(0))$. In particular, in the process of handling \eqref{eq 5.20}, the term $\frac{m}{P_m} \ka_i^{m-1}\z[(2-\es)\frac{(S_k)_i^2}{S_k}-S_k^{pp,qq}u_{ppi}u_{qqi}\x]$ appears. As mentioned in Subsection \ref{subsec 4.2}, the treatment of this term is analogous to that of $A_i$ in Section \ref{sec 3}, when $\es$ is sufficiently small. Therefore, by choosing the constants that satisfy \eqref{constants condition 3}, we have the following inequality similar to inequality \eqref{eq 5.11}:
    $$
    -C\frac{\ga_0}{u}\geq-\frac{C}{u_{11}}+\frac{c_0}{2} \si_1^{\frac{1}{k-2}}\geq  -\frac{C}{u_{11}}+\frac{c_0}{2} u_{11}^{\frac{1}{k-2}},
    $$
    where $\ga_0$ is a sufficiently large constant. Hence, we complete the proof .
\end{proof}

\begin{proof}[\textbf{Proof of Theorem $\ref{thm rigidity theorem 2}$}]
    Assume $u$ is an entire solution of the equation 
    $$
    -u_t\cdot (\si_k(D^2 u)+\al \si_{k-1}(D^2 u))=1, \quad \te{in } \rr^n \times (-\wq,0].    
    $$
    For arbitrary positive constant $R>1$, we define
    $$
    v(y,t):=\frac{u(Ry,R^2 t )-R^2}{R^2}, \quad  \mathcal{O}_R:=\{(y,t): u(Ry,R^2 t )\leq R^2\}.
    $$
    Then, $v$ satisfies the following Dirichlet problem:
    \begin{equation*}
        \left\{\begin{aligned}
            & -v_t \cdot S_k(D^2 v)=1, & & \te{in} \ \mathcal{O}_R, \\
            & v=0, & & \te{on} \ \pa\mathcal{O}_R.
        \end{aligned}\right.
    \end{equation*}
    By Lemma \ref{lem 5.2}, we have the following estimates
    \begin{equation}\label{eq 5.22}
        (-v)^{\ga_0}\dd v\leq C,
    \end{equation}
    where the constants $\ga_0$ and $C$ depend on $n$, $k$ $m_1$, $m_2$ and $\operatorname{diam}(\mathcal{O}(0))$. By the quadratic growth condition, it follows that
    $$
    a|Ry|^2-b\leq  u(Ry,0)\leq u(Ry,R^2t)\leq R^2, \quad \te{in} \ \mathcal{O}_R.
    $$
    Namely
    $$
    |y|^2 \leq \frac{1+b}{a}.
    $$
    Hence, the set $\{|x|:(x,t)\in \mathcal{O}_R\}$ is bounded. Thus, the constant $\ga_0$ and $C$ are independent of $\operatorname{diam}(\mathcal{O}_R(0))$. Next, we consider the domain
    $$\mathcal{O}_R'=\{(y,t): u(Ry,R^2 t)\leq R^2/2\} \subset \mathcal{O}_R.$$
    In $\mathcal{O}_R'$, we have
    $$
    v(y,t)\leq-\frac{1}{2}.
    $$
    Hence, it follows form \eqref{eq 5.22} that 
    \begin{equation*}
        \dd v\leq 2^{\ga_0} C, \quad \te{in} \ \mathcal{O}_R'.
    \end{equation*}
    Hence $(\dd u)(Ry,R^2 t)=\dd v(y,t)\leq 2^{\ga_0}C$ for $(y,t) \in \mathcal{O}_R'$. That means
    $$
    \dd u(x,t)\leq C_0, \quad \te{in} \ \mathcal{O}_R''=\{(x,t):u(x,t)\leq R^2/2\},
    $$
    where $C_0$ is an absolutely constants. The estimate above holds true for all large $R$. Thus we using Evans-Krylov theory \cite{NAKAMORI2015211}, we have for any $0<\bb<1$,
    $$
    \|D^2 u\|_{C^{0,\bb}(\mathcal{O}_R''\cap B_{R/2})}\leq C \frac{\|D^2 u\|_{L^\wq(\mathcal{O}_R''\cap B_R)}}{R^\bb}\leq\frac{C }{R^\bb}.
    $$
    Letting $R\to \wq$, and we finish the proof of Theorem \ref{thm rigidity theorem 2}.
\end{proof}

\section*{Acknowledgments}

The authors would like to thank Professor Xi-Nan Ma for the constant
encouragement in this subject, and also thank Xinqun Mei for his helpful discussion and encouragement.

\end{document}